\documentclass[12pt]{birkjour_t2}
\usepackage{amsfonts}
\usepackage{bbm}
\usepackage{txfonts}
\usepackage{amssymb,amsmath,amsthm,mathrsfs}
\numberwithin{equation}{section}

\newtheorem{thm}{Theorem}[section]
\newtheorem{lemma}[thm]{Lemma}
\newtheorem{remark}[thm]{Remark}

\newtheorem{definition}[thm]{Definition}

\newcommand{\mb}{\mathbf}
\newcommand{\bx}{{\mb X}}
\newcommand{\ba}{{\mb A}}
\newcommand{\bw}{{\mb W}}
\newcommand{\bi}{{\mb I}}
\newcommand{\bb}{{\mb B}}
\newcommand{\bq}{{\mb Q}}
\newcommand{\be}{{\mb e}}

\newcommand{\re}{{\rm E}}
\newcommand{\E}{{\rm E}}
\newcommand{\rtr}{{\rm tr}}
\renewcommand{\(}{\left(}
\renewcommand{\)}{\right)}
\newcommand{\lj}{\left|}
\newcommand{\rj}{\right|}

\newcommand{\bd}{\mb D}
\newcommand{\vk}{\varepsilon_{k}(z)}
\newcommand{\zk}{\zeta_{k}(z)}

\allowdisplaybreaks

\begin{document}

\title[Gaussian fluctuations for linear statistics of Wigner ensembles] {Gaussian fluctuations for linear spectral statistics of Wigner beta ensembles}

\author{Yanqing, Yin}

\address{}
\email{yinyq799@nenu.edu.cn}


\maketitle

\begin{abstract}
As an important topic in Mathematical Physics and statistics, random matrices theory has found uses in many aspects of modern physics and multivariate analysis. This paper is to investigate the Gaussian fluctuations for linear spectral statistics (LSS) of  Wigner beta ensembles. We first establish a central limit theorem (CLT) for LSS of Wigner quaternion matrices, then give a general CLT for Wigner $\beta$ ensembles.


\end{abstract}

\section{Introduction}
Random matrices theory is known as an important topic in Mathematical Physics. It is shown to be inter-related with log-gases and the Calogero-Sutherland model. As an early introduced matrix models, the Gaussian $\beta$ ensembles have been considered by a large number of authors. Here the special cases $\beta=1,2,4$, known as Dyson's three-fold-way \cite{dyson1962threefold}, correspond to Gaussian Orthogonal Ensemble (GOE),
Gaussian Unitary Ensemble (GUE) and Gaussian Symplectic Ensemble (GSE) respectively. And the entries of a certain matrix in the above three ensembles are real, complex and quaternion standard gaussian variables. Since we can compute the explicit density functions of the joint distribution of eigenvalues, lots of properties of $G(O/U/S)E$ have been deduced by means of orthogonal polynomial.  More details can be found in \cite{mehta2004random}.
On the other hand, the central concept of the random matrix theory, as envisioned by E. Wigner, is the hypothesis
that the distributions of eigenvalue spacings of large complicated quantum systems are
universal in the sense that they depend only on the symmetry classes of the physical systems but
not on detailed structures. This concept is also called ``universality". By dropping the gaussian assumption, one consider the more general ensembles, the so called Wigner ($\beta$) ensembles where $\beta=1,2,4$ (or equivalently, the wigner real, complex, quaternion ensembles). For details in this direction, we refer the reader to \cite{tao2010random,Bourgade2014,Johansson2001,Soshnikov1999} and references therein.

Also known as central limit theorems, global fluctuations for linear statistics have been of interest to the
random matrix community for a long time. A variety of models and eigenvalue distributions have
been studied from this point of view \cite{Girko2001,Anderson2006,Sosoe2013,Benaych-Georges2014,Lytova2009,Anderson2008,bai2004clt,Bai2005}. In this paper, as an extension of the results in \cite{Bai2005}, we will show the CLT for linear statistics of Wigner quaternion ensemble and thus fulfilling the corresponding CLT for Wigner ($\beta$) ensembles.

\section{Some Definitions And Main Results}
We begin by a list of definitions and background that will be used in this paper.
Set an ordered basis
\begin{displaymath}
\mathbbm e =\mb I_2= \left( \begin{array}{cc}
1&0\\
0&1\\
\end{array} \right), \quad
\mathbbm{i} = \left( \begin{array}{cc}
i&0\\
0&- i\\
\end{array} \right), \quad
\mathbbm j = \left( \begin{array}{cc}
0&1\\
-1&0\\
\end{array}\right), \quad
\mathbbm{k} = \left( \begin{array}{cc}
0&i\\
i&0\\
\end{array}\right),
\end{displaymath}
where $i=\sqrt{-1}$ denotes the usual imaginary unit (here and in the rest of the paper, $\mb I_n$ denote the $n$ dimensional identity matrix), then a quaternion can be represented as
$$q = a \cdot \mathbbm e + b \cdot \mathbbm i + c \cdot \mathbbm j + d \cdot \mathbbm k = \left( {\begin{array}{*{20}{c}}
\alpha &\beta \\
{ - \overline \beta }&{\overline \alpha }
\end{array}} \right), $$
where $a, b, c, d$ are  real  and $\alpha  = a + b i$, $\beta = c + di$ are complex.
The quaternion conjugate of $q$ is defined as
$$q^Q = a \cdot \mathbbm  e - b \cdot \mathbbm  i - c \cdot \mathbbm  j - d \cdot \mathbbm  k = \left( {\begin{array}{*{20}{c}}
{\bar \alpha }&{ - \beta }\\
{\bar \beta }&\alpha
\end{array}} \right)=q^*,$$
where $\(\cdot\)^*$ denote conjugate transform of a matrix.

We also write
$$qq^*=q^*q=\({a^2+b^2+c^2+d^2}\)\mb I_2\triangleq\|q\|_{Q}^2\mb I_2.$$

A Wigner quaternion matrix of size $n$ is a quaternion self-dual Hermitian matrix where the upper-triangle entries  are independent quaternion random variables. From \cite{zhang1997}, we know that a quaternion Hermitian matrix has 2n pairwise real eigenvalues. Suppose $\lambda_1^{\mb Q},\lambda_1^{\mb Q},\cdots,\lambda_n^{\mb Q},\lambda_n^{\mb Q}$ are the 2n real eigenvalues of a $n\times n$ quaternion self-dual Hermitian matrix (a $2n\times 2n$ Hermitian matrix) $\mb Q$, then we call $\lambda_1^{\mb Q}\mb I_2,\lambda_2^{\mb Q}\mb I_2,\cdots,\lambda_n^{\mb Q}\mb I_2$ are the n quaternion eigenvalues of $\mb Q$.

In this paper, we define
\begin{align*}
\E(\|q\|_Q^k)&=\E\(\(\sqrt{a^2+b^2+c^2+d^2}\)^k\).
\end{align*}
as the k-th norm moment of the quaternion random variable $q$.

For any function of bounded variation $G$ on the real line, its Stieltjes transform is defined by
$$m_G(z)=\int\frac{1}{y-z}dG(y),~~z\in\mathbb{C}^{+}\equiv\{z\in\mathbb{C}:\Im z>0\}.$$

In \cite{Yin2016}, it is shown that under a Lindeberg type condition as $n\to \infty$, the Empirical Spectral Distribution (ESD) of a Wigner quaternion matrix whose entries being zero means and unit variances, convergence to the  standard semicircular law $F$ with density function
\begin{align*}
F'(x)=\begin{cases}
\frac1{2\pi}\sqrt{4-x^2},& \ {\rm if}\ |x|\le2\\
0,\quad&{\rm otherwise}.
\end{cases}
\end{align*}

Below are definitions for two kinds of matrices that will be used in the proof of this paper.
\begin{definition}
A matrix is called Type-T matrix if it has the following structure:$$\left( {\begin{array}{*{20}{c}}
t&0\\
0&t
\end{array}} \right).$$
\end{definition}

\begin{definition} A matrix is called Type-\uppercase\expandafter{\romannumeral1} matrix if it has the following structure:
\[\left( {\begin{array}{*{20}{c}}
{{t_1}}&0&{{a_{12}}}&{{b_{12}}}& \cdots &{{a_{1n}}}&{{b_{1n}}}\\
0&{{t_1}}&{{c_{12}}}&{{d_{12}}}& \cdots &{{c_{1n}}}&{{d_{1n}}}\\
{{d_{12}}}&{ - {b_{12}}}&{{t_2}}&0& \cdots &{{a_{2n}}}&{{b_{2n}}}\\
{ - {c_{12}}}&{{a_{12}}}&0&{{t_2}}& \cdots &{{c_{2n}}}&{{d_{2n}}}\\
 \vdots & \vdots & \vdots & \vdots & \ddots & \vdots & \vdots \\
{{d_{1n}}}&{ - {b_{1n}}}&{{d_{2n}}}&{ - {b_{2n}}}& \cdots &{{t_n}}&0\\
{ - {c_{1n}}}&{{a_{1n}}}&{ - {c_{2n}}}&{{a_{2n}}}& \ldots &0&{{t_n}}
\end{array}} \right).\]
\end{definition}
We note that in this paper, we will use $\mb 0$ denote a two dimensional zero matrix. We also use $C$ to stand for a constant that may take different values from one appearance to others.

Let $\mu(f)$ denote the integral of a function $f$ with respect to a signed measure $\mu$ and $\mathcal{U}$ be an open set of the complex plane that contains the interval $[-2,2]$ (the support of the standard semicircular law $F$). Define $\mathcal{A}$ to be the set of analytic functions $f: \mathcal{U}\mapsto\mathbb{C}$ and $F_n$ to be the empirical spectral distribution (ESD) of a wigner $\beta$ matrix $\bw_n$. We then consider the empirical process $G_n:=\{G_n(f)\}$ indexed by $\mathcal{A}$, i.e.,
\begin{align}\label{al3}
G_n(f)=n\int f(x)d\(F_n(x)-F(x)\), \quad f\in\mathcal{A}.
\end{align}

Now, we are in position to present our main theorem.

Define $\{T_k\}$ to be the family of Tchebychev polynomials and
\begin{align}
\psi_l(f)=\frac{1}{2\pi}\int_{-\pi}^{\pi}f(2\cos(\theta))\exp^{il\theta}d\theta=\frac{1}{\pi}\int_{-1}^1f(2s)T_l(s)\frac{1}{\sqrt{1-s^2}}ds.
\end{align}
for any integer $l\geq 0$. We have the following theorem:
\begin{thm}\label{th1}
Assume that $\bw_n=n^{-1}\bx_n=n^{-1}(x_{jk})$ is a Wigner quaternion matrix,
\begin{itemize}
\item[(a)] For all $j$, $\re\|x_{jj}\|_Q^2=\sigma^2$, for $j<k$, $\re\|x_{jk}\|_Q^2=1$, and for $j\le k$, $\re x_{jk}=\mb 0$.
\item[(b)] For $j\neq k$, write $x_{jk} =\left(
        \begin{array}{cc}
          \alpha_{jk} & \beta_{jk} \\
          -\bar \beta_{jk} & \bar \alpha_{jk} \\
        \end{array}
      \right),
$
then $\re\|x_{jk}\|_Q^4=M$ and
\begin{align*}
&\re|\alpha_{jk}|^2=1/2,\quad \re|\beta_{jk}|^2=1/2,\quad\re\alpha_{jk}^2=0,\\
 &\re\beta_{jk}^2=0,\quad\re(\alpha_{jk}\beta_{jk})=\re(\alpha_{jk}\bar\beta_{jk})=0.
 \end{align*}

\item[(c)] For any $\eta>0$, as $n\to\infty$,
\begin{align*}
\frac1{\eta^4n^2}\sum_{j,k}\re\(\|x_{jk}\|_Q^4I\(\|x_{jk}\|_Q\ge\eta\sqrt n\)\)\to0.
\end{align*}
\end{itemize}
Then the spectral empirical process $G_n=G_n(f)$ indexed by the set of analytic functions $\mathcal{A}$ converges weakly in finite dimension to a Gaussian process $G:=\{G(f):f\in\mathcal{A}\}$ with mean function $\re G(f)$ given by
\begin{align*}
\frac{-\{f(2)+f(-2)\}}{8}+\frac{1}{4}\psi_0(f)+\(\sigma^2-\frac12\)\psi_2(f)+\(M-\frac32\)\psi_4(f)
\end{align*}
and the covariance function $\mb Cov(f,g):=\re\(G(f)-\re G(f)\)\(G(g)-\re G(g)\)$ given by
\begin{align*}
\sigma^2\psi_1(f)\psi_1(g)+(2M-1)\psi_2(f)\psi_2(g)+\frac{1}{2}\sum_{l=3}^{\infty}l\psi_l(f)\psi_l(g)\\
=\frac{1}{4\pi^2}\int_{-2}^2\int_{-2}^2f'(t)g'(s)V(t,s)dtds,
\end{align*}
where
\begin{align*}
V(t,s)=\(\sigma^2-\frac{1}{2}+\(\frac M2-\frac34\)ts\)\sqrt{\(4-t^2\)\(4-s^2\)}+\frac12\log\(\frac{4-ts+\sqrt{\(4-t^2\)\(4-s^2\)}}{4-ts-\sqrt{\(4-t^2\)\(4-s^2\)}}\).
\end{align*}
\end{thm}

Combining with the existing results \cite{Bai2005}, we shall establish the following general CLT.

\begin{thm}\label{thz}
Assume that $\bw_n=n^{-1}\bx_n=n^{-1}(x_{jk})$ is a Wigner $\beta$ matrix,
\begin{itemize}
\item[(a)] For all $j$,
\begin{align*}
  \left\{
            \begin{array}{ll}
              \re|x_{jj}|^2=\sigma^2, & \hbox{$\beta$=1,2;} \\
              \re\|x_{jj}\|_Q^2=\sigma^2, & \hbox{$\beta$=4,}
            \end{array}
          \right.
\end{align*}
for $j<k$,
\begin{align*}
  \left\{
            \begin{array}{ll}
              \re|x_{jk}|^2=1, \quad \re|x_{jk}|^4=M, & \hbox{$\beta$=1,2;} \\
              \re\|x_{jk}\|_Q^2=1, \quad \re\|x_{jk}\|_Q^4=M, & \hbox{$\beta$=4,}
            \end{array}
          \right.
\end{align*}
and for $j\le k$
\begin{align*}
  \left\{
            \begin{array}{ll}
              \re x_{jk}=0, & \hbox{$\beta$=1,2;} \\
              \re x_{jk}=\mb 0, & \hbox{$\beta$=4.}
            \end{array}
          \right.
\end{align*}

\item[(b)] For $j\neq k$, $\beta=2$, $\E x_{j,k}^2=0$.
For $j\neq k$, $\beta=4$, write
$x_{jk} =\left(
        \begin{array}{cc}
          \alpha_{jk} & \beta_{jk} \\
          -\bar \beta_{jk} & \bar \alpha_{jk} \\
        \end{array}
      \right),
$
then
\begin{align*}
&\re|\alpha_{jk}|^2=1/2,\quad \re|\beta_{jk}|^2=1/2,\quad\re\alpha_{jk}^2=0,\\
 &\re\beta_{jk}^2=0,\quad\re(\alpha_{jk}\beta_{jk})=\re(\alpha_{jk}\bar\beta_{jk})=0.
 \end{align*}
\item[(c)] For any $\eta>0$, as $n\to\infty$,
\begin{align*}
  \left\{
            \begin{array}{ll}
             \frac1{\eta^4n^2}\sum_{j,k}\re\(|x_{jk}|^4I\(|x_{jk}|\ge\eta\sqrt n\)\)\to0, & \hbox{$\beta$=1,2;} \\
              \frac1{\eta^4n^2}\sum_{j,k}\re\(\|x_{jk}\|_Q^4I\(\|x_{jk}\|_Q\ge\eta\sqrt n\)\)\to0, & \hbox{$\beta$=4.}
            \end{array}
          \right.
\end{align*}

\end{itemize}
Then the spectral empirical process $G_n=G_n(f)$ indexed by the set of analytic functions $\mathcal{A}$ converges weakly in finite dimension to a Gaussian process $G:=\{G(f):f\in\mathcal{A}\}$ with mean function $\re G(f)$ given by
\begin{align*}
\(\frac{2}{\beta}-1\)\(\frac{\{f(2)+f(-2)\}}{4}-\frac{\psi_0(f)}{2}\)+\(\sigma^2-\frac{2}{\beta}\)\psi_2(f)+\(M-1-\frac{2}{\beta}\)\psi_4(f)
\end{align*}
and the covariance function $\mb Cov(f,g):=\re\(G(f)-\re G(f)\)\(G(g)-\re G(g)\)$ given by
\begin{align*}
\sigma^2\psi_1(f)\psi_1(g)+2(M-\frac{2}{\beta})\psi_2(f)\psi_2(g)+\frac{2}{\beta}\sum_{l=3}^{\infty}l\psi_l(f)\psi_l(g)\\
=\frac{1}{4\pi^2}\int_{-2}^2\int_{-2}^2f'(t)g'(s)V(t,s)dtds,
\end{align*}
where
\begin{align*}
V(t,s)=\(\sigma^2-\frac{2}{\beta}+\(M-1-\frac{2}{\beta}\)\frac{ts}{2}\)\sqrt{\(4-t^2\)\(4-s^2\)}+\frac{2}{\beta}\log\(\frac{4-ts+\sqrt{\(4-t^2\)\(4-s^2\)}}{4-ts-\sqrt{\(4-t^2\)\(4-s^2\)}}\).
\end{align*}
\end{thm}

\section{Proof of Theorem \ref{th1}}

Let $\mathcal{C}$ be the contour made by the boundary for the rectangle with vertices $(\pm a\pm iv_0)$, where $a>2$ and $1\ge v_0>0$. We can always assume that the constants $a-2$ and $v_0$ are sufficiently small so that $\mathcal{C}\subset\mathcal{U}$.

By Cauchy integral formula, we have
\begin{align*}
G_n(f)=-\frac1{2\pi i}\oint_{\mathcal{C}}n\left[m_n(z)-m(z)\right]dz
\end{align*}
where $m_n(z)$ and $m(z)$ are Stieltjes transform of $\bw_n$ and the semicircular law $F$, respectively. The equality above may not be correct when some eigenvalues of $\bw_n$ run outside the contour. Thus we need to consider the corrected version, i.e.
\begin{align*}
G_n(f)I(B_n^c)=-\frac1{2\pi i}I(B_n^c)\oint_{\mathcal{C}}n\left[m_n(z)-m(z)\right]dz
\end{align*}
where $B_n=\{|\lambda_{ext}(\bw_n)|\ge1+a/2\}$ and $\lambda_{ext}$ denotes the smallest or largest eigenvalue of the matrix $\bw_n$. Notice that in \cite{yin2014eigen} it is shown that after truncation and renormalization, for any $a>2$ and $t>0$,
\begin{align}\label{al4}
{\rm P}(B_n)=o(n^{-t}),
\end{align}
we know that this difference will not matter in the proof. The mentioned representation reduces the proof of  Theorem \ref{th1} to showing that the process $M_n:=\{M_n(z),z\notin[-2,2]\}$, where
\begin{align*}
M_n(z)=n\left[m_n(z)-m(z)\right],
\end{align*}
converges to a Gaussian process $M(z)$, $z\notin[-2,2]$. We shall show this conclusion by establishing the following theorem.

\begin{thm}\label{th2}
Under conditions in Theorem \ref{th1}, the process $\{M_n(z);\mathbb{C}_0\}$ where $\mathbb{C}_0=\{z=u+iv:|v|\ge v_0\}$, converges weakly to a Gaussian process $\{M(z);\mathbb{C}_0\}$ satisfying for $z\in\mathbb{C}_0,$
\begin{align*}
\re M(z)=\(1+m'(z)\){m^3(z)}\({\sigma^2}-1-\frac12m'(z)+\(M-\frac{3}2\)m^2(z)\)
\end{align*}
and, for $z_1,z_2\in\mathbb{C}_0$,
\begin{align*}
{\rm Cov}\(M(z_1),M(z_2)\)=&m'(z_1)m'(z_2)\Bigg(\sigma^2-1/2+\({2M-{3}}\)m(z_1)m(z_2)\\
&\quad\quad+\frac1{2\(1-{m(z_1)m(z_2)}\)^2}\Bigg).
\end{align*}

\end{thm}

\begin{remark}
  The process $\{M(z);\mathbb{C}_0\}$ in Theorem \ref{th2} can be taken as a restriction of a process $\{M(z)\}$ defined on the whole complex plane except the real axis since the mean and covariance functions of $M(z)$ are independent of $v_0$. Then, by the symmetry that $M(\bar z)=\overline{M(z)}$ and the continuity of the mean and covariance functions of $M(z)$ , one may extend the process to $\{M(z);\Re z\notin[-2,2]\}$.
\end{remark}

Define a slowly varying sequence of positive constants $\varepsilon_n$ that convergence to 0. Split the contour $\mathcal{C}$ as the union $\mathcal{C}_u+\mathcal{C}_l+\mathcal{C}_r+\mathcal{C}_0$, where
\begin{align*}
\left\{
         \begin{array}{ll}
           \mathcal{C}_l=\{z=-a+iy:\varepsilon_nn^{-1}<|y|\le v_0\},  \\
           \mathcal{C}_r=\{z=a+iy:\varepsilon_nn^{-1}<|y|\le v_0\},  \\
           \mathcal{C}_0=\{z=\pm a+iy,|y|\le\varepsilon_nn^{-1}\},
         \end{array}
       \right.
\end{align*}
by Theorem \ref{th2}, we get the weak convergence
$
\int_{\mathcal{C}_u}M_n(z)dz\Rightarrow\int_{\mathcal{C}_u}M(z)dz.
$
To prove Theorem \ref{th1}, it is sufficient to show that, for $j=l,r,0$,
\begin{align}\label{al1}
\lim_{v_0\downarrow0}\limsup_{n\to\infty}\re\lj\int_{\mathcal{C}_j}M_n(z)I\(B_n^c\)dz\rj=0
\end{align}
and
\begin{align}\label{al2}
\lim_{v_1\downarrow0}\re\lj\int_{\mathcal{C}_j}M(z)dz\rj=0.
\end{align}
Estimate (\ref{al2}) can be verified directly by the mean and variance functions of $M(z)$. The proof of (\ref{al1}) will be postponed to subsection \ref{jl0}.

\subsection{Truncation and Renormalization}

Note that condition $(c)$ in Theorem \ref{th1} implies the existence of a sequence $\eta_n\downarrow0$ such that
\begin{align*}
\frac1{\eta_n^4n^2}\sum_{j,k}\re\left[\|x_{jk}\|_Q^4I\(\|x_{jk}\|_Q\ge\eta_n\sqrt n\)\right]\to0.
\end{align*}
Here $\eta_n\to0$ may be assumed to be as slow as desired. For definiteness, we assume that $\eta_n>1/\log n$.

At first, truncate the variables as $\hat x_{jk}=x_{jk}I\(\|x_{jk}\|_Q\le\eta_n\sqrt n\)$. Then normalize them by setting $\tilde x_{jk}=\(\hat x_{jk}-\re\hat x_{jk}\)/\sigma_{jk}$ for $j\neq k$ and $\tilde x_{jj}=\sigma\(\hat x_{jj}-\re\hat x_{jj}\)/\sigma_{jj}$, where $\sigma_{jk}$ is the standard deviation of $\hat x_{jk}$. Let $\widehat F_n$ and $\widetilde F_n$ be the ESD of the random matrices $\widehat\bw_n=\(\frac1{\sqrt n}\hat x_{jk}\)$ and $\widetilde\bw_n=\(\frac1{\sqrt n}\tilde x_{jk}\)$, respectively. According to (\ref{al3}), we similarly define $\widehat G_n$ and $\widetilde G_n$.

To begin with, we have
\begin{align*}
{\rm P}\(F_n\neq\widehat F_n\)\le&\sum_{j,k}{\rm P}\(\|x_{jk}\|_Q>\eta_n\sqrt n\)\\
\le&\frac1{\eta_n^4n^2}\sum_{j,k}\re\left[\|x_{jk}\|_Q^4I\(\|x_{jk}\|_Q>\eta_n\sqrt n\)\right]\to0
\end{align*}
which implies
\begin{align*}
{\rm P}\(G_n\neq\widehat G_n\)\le{\rm P}\(F_n\neq \widehat F_n\)=o(1).
\end{align*}
Next, we will compare $\widetilde G_n$ with $\widehat G_n$. Denote by $\hat\lambda_{nj}$ and $\tilde\lambda_{nj}$ the $j$th largest eigenvalues of $\widehat\bw_n$ and $\widetilde\bw_n$, respectively. Using Lemma \ref{cll1}, we have
\begin{align*}
\re\lj\widehat G_n(f)-\widetilde G_n(f)\rj^2=&\re\lj\sum_{j=1}^n\(f(\hat\lambda_{nj})-f(\tilde\lambda_{nj})\)\rj^2=\re\lj\sum_{j=1}^n f'(\xi_j)\(\hat\lambda_{nj}-\tilde\lambda_{nj}\)\rj^2\\
\le& Cn\re\sum_{j=1}^n\lj\hat\lambda_{nj}-\tilde\lambda_{nj}\rj^2\le Cn\re\sum_{jk}\lj\lj n^{-1/2}(\hat x_{jk}-\tilde x_{jk})\rj\rj_Q^2\\
\le&C\bigg[\sum_{j\neq k}\(|1-\sigma_{jk}^{-1}|^2\re\|\hat x_{jk}\|_Q^2+\|\re \hat x_{jk}\|_Q^2\sigma_{jk}^{-2}\)\\
&+\sum_{j}\(|1-\sigma\sigma_{jj}^{-1}|^2\re\|\hat x_{jj}\|_Q^2+\|\re \hat x_{jj}\|_Q^2\sigma^2\sigma_{jj}^{-2}\)\bigg]\\
\le&C\sum_{jk}\(n^{-2}\eta_n^{-4}+n^{-3}\eta_n^{-6}\)\re^2\|x_{jk}\|_Q^4I\(\|x_{jk}\|_Q>\eta_n\sqrt n\)\to0.
\end{align*}

Therefore, we conclude that
\begin{align*}
\widetilde G_n(f)=G_n(f)+o_p(1).
\end{align*}
This yields that we only need to find the limiting distribution of $\{\widetilde G_n(f_j),j=1,\cdots,\kappa\}$. Hence, in what follows, we shall assume the underlying variables are truncated at $\eta_n\sqrt n$, centralized, and renormalized. For simplicity, we shall suppress all sub- or superscripts on the variables and  assume that
$$\|x_{jk}\|_Q\le\eta_n\sqrt n, \ \re x_{jk}=0, \ \re\|x_{jk}\|_Q^2=1, \ {\rm and} \ \re\|x_{jk}\|_Q^4=+o(1).$$

\subsection{Mean function of $M_n(z)$}

Let $\bq_k =(x_{1k}',\cdots,x_{(k-1)k}',x_{(k+1)k}',\cdots,x_{nk}')_{(2n-2)\times2}'$ denote the $k$th quaternion column of $\bx_n$ with $k$th quaternion elements removed. Let $\bw_{nk}$ be the matrix obtained from $\bw_n$ with the $k$th quaternions column and row removed. Moreover, write
\begin{align*}
&\bd(z)= (\bw_n-z\bi_{2n})^{-1}, \ \bd_k(z)=(\bw_{nk}-z\bi_{2n-2})^{-1},\\
&\vk = n^{-1/2}x_{kk}-n^{-1}\bq_k^*\bd_k\bq_k + \re m_n(z)*\bi_2,\\
&\zk= \((z + \re m_n(z)) *\bi_2-\vk\)^{-1}, \ t_n(z)=\(z + \re m_n(z)\)^{-1}.
\end{align*}

Yin, Bai and Hu in \cite{Yin2016} derived
\begin{align}\label{al16}
\re m_n(z)=&-\frac12\(z-\delta_n(z)-\sqrt{(z+\delta_n(z))^2-4}\)\\
=&\delta_n(z)+m\(z+\delta_n(z)\)\notag
\end{align}
where
\begin{align*}
\delta_n(z) = -\frac{t_n(z)}{2n}\sum_{k=1}^n\re\rtr\(\vk\zk\).
\end{align*}
Hence, for $z\in\mathbb{C}_0$ we have
\begin{align*}
\re M_n(z)=\(1+m'(z)\)\(1+o(1)\)n\delta_n(z)
\end{align*}
and
\begin{align*}
n\delta_n(z) =-\frac{t_n(z)}{2}\sum_{k=1}^n\re\rtr\(\vk\zk\).
\end{align*}
This yields that it is suffices to show the limit of $n\delta_n(z)$ for $z\in\mathbb{C}_0$. Here we show a stronger result that the limit of $n\re\delta_n(z)$ holds uniformly in $z\in \mathcal{C}_n=\mathcal{C}_u+\mathcal{C}_l+\mathcal{C}_r$.

To begin with, we claim that the moments of $\|\bd(z)\|$, $\|\bd_k(z)\|$, and $\|\bd_{k,j}(z)\|$ are bounded in n and $z\in\mathcal{C}_n$. Without loss of generality, we only give the proof for $\re\|\bd_1(z)\|^j$ and the others are similar.  In fact,
\begin{align*}
\|\bd_1(z)\|=&\|\bd_1(z)\|I(B_n^c)+\|\bd_1(z)\|I\(B_n\)\\
\le&\max\{\frac2{a-2},\frac1{v_0}\}+n\varepsilon_n^{-1}I\(B_n\)\le C+n^2I\(B_n\)\notag.
\end{align*}
 Then using (\ref{al4}), we have for any positive $j$ and suitably large $t$
\begin{align*}
\re\|\bd_1(z)\|^j
\le&C_1+C_2n^{2j}n^{-t}\le C_j.
\end{align*}

Note that
\begin{align}\label{al17}
n\delta_n(z) =&-\frac{t_n^2(z)}{2}\sum_{k=1}^n\re\rtr\vk-\frac{t_n^2(z)}{2
}\sum_{k=1}^n\re\rtr\(\varepsilon_k^2(z)\zk\)\\
 =&-\frac{t_n^2(z)}{2}\sum_{k=1}^n\re\rtr\vk-\frac{t_n^3(z)}{2
}\sum_{k=1}^n\re\rtr\varepsilon_k^2(z)-\frac{t_n^3(z)}{2
}\sum_{k=1}^n\re\rtr\(\varepsilon_k^3(z)\zk\)\notag\\
\triangleq&\mathcal{I}_1+\mathcal{I}_2+\mathcal{I}_3.\notag
\end{align}
We begin with the estimation of $\mathcal{I}_3$. From \cite{yin2014rate}, it can be verified that
\begin{align}\label{al10}
|\tilde t_n(z)|<1
\end{align}
along the same line, where $\tilde t_n(z)=\(z+m_n(z)\)^{-1}$. Using Lemma \ref{cll3}, we get
\begin{align*}
\rtr\zeta_k^{-1}(z)=4\left[\rtr\zeta_k(z)\right]^{-1}.
\end{align*}
Suppose $\lj\rtr\zeta_k(z)\rj>4$, then
\begin{align*}
\lj\rtr\tilde\varepsilon_k(z)\rj=\lj\rtr\(\zeta_k^{-1}(z)-\tilde t_n^{-1}(z)\bi_2\)\rj>1
\end{align*}
where
\begin{align*}
\tilde\varepsilon_k(z)=\frac1{\sqrt n}x_{kk}-\frac1n\bq_k^*\bd_k(z)\bq_k+m_n(z)\bi_2.
\end{align*}

Note that
\begin{align}\label{al5}
\left\|\frac1{\sqrt n}x_{kk}\right\|_Q\le\eta_n\to0
\end{align}
and
\begin{align}\label{al6}
&\lj\rtr\bd_k^{-1}(z)-\rtr\bd^{-1}(z)\rj I_{B_n^c}\\
\le&2\left[\sum_{k=1}^{n-1}\frac{|\lambda_j-\lambda_{kj}|}{|\(\lambda_j-z\)\(\lambda_{kj}-z\)|}+\frac1{|\lambda_n-z|}\right]I_{B_n^c}\notag\\
\le&2\max\left\{v_0^{-1},2/(a-2)\right\}^2\left[\sum_{k=1}^{n-1}\(\lambda_j-\lambda_{kj}\)+1\right]I_{B_n^c}\notag\\
\le&C[\lambda_1-\lambda_n+1]I(B_n^c)\le C(a+3)\notag
\end{align}
where $\lambda_j$ and $\lambda_{kj}$ are the eigenvalues of $\bw_n$ and $\bw_{nk}$ in decreasing order, respectively.

From (3.22) in \cite{Yin2016}, it is known that
\begin{align}\label{al11}
\lj\rtr\zeta_k\rj\le2v^{-1}.
\end{align}
If $|\rtr H(z)|\le n^{\theta}$ uniformly in $z\in\mathcal{C}_n$ for some $\theta>0$, where $H(z)$ is a $2\times 2$ matrix, then by (\ref{al11}), Lemma \ref{cll3}, and Lemma \ref{cll4}, we get for suitably large $t$ and $l$
\begin{align}\label{al8}
&\re\lj\rtr\big(\zeta_k(z)H(z)\big)\rj=\frac12\re\lj\rtr\zeta_k(z)\rtr H(z)\rj\\
\le&2\re\lj\rtr H(z)\rj+n^{\theta}v^{-1}{\rm P}\(\lj\rtr\zeta_k(z)\rj>4\)\notag\\
\le&2\re\lj\rtr H(z)\rj+n^{2+\theta}{\rm P}\(\lj\rtr\tilde\varepsilon_k(z)\rj>1\)\notag\\
\le&2\re\lj\rtr H(z)\rj+n^{2+\theta}{\rm P}\(\lj\rtr\(\bq_k^*\bd_k\bq_k\)-\rtr\bd_k\rj>\frac n2,B_{nk}^c\)+n^{2+\theta}{\rm P}\(B_{nk}\)\notag\\
\le&2\re\lj\rtr H(z)\rj+Cn^{2+\theta-l}\re\lj\rtr\(\bq_k^*\bd_k\bq_k\)-\rtr\bd_k\rj^lI\(B_{nk}^c\)+n^{-t}\notag\\
\le&2\re\lj\rtr H(z)\rj+Cn^{2+\theta-l}\(n^{l/2}+n^{l-1}\eta_n^{2l-4}\)+n^{-t}\le2\re\lj\rtr H(z)\rj+o(n^{-t})\notag
\end{align}
uniformly in $z\in\mathcal{C}_n$.

Now, let us apply the above inequality to prove $\mathcal{I}_3=o(1)$ uniformly in $z\in\mathcal{C}_n$. Choose $H(z)=\varepsilon_k^3(z)$. By Lemma \ref{cll3}, we have
\begin{align}\label{al12}
|\rtr \varepsilon_k^3(z)|=\frac14|\rtr\varepsilon_k(z)|^3\le 2\eta_n^3+\frac{2n^3\eta_n^6}{v^3}+\frac2{v^3}\le Cn^9.
\end{align}
Hence, applying (\ref{al8}), one only needs to prove that
\begin{align*}
\re|\rtr\varepsilon_k(z)|^3=o(1)
\end{align*}
uniformly in $z\in\mathcal{C}_n$ and $k\le n$.

Using Lemma \ref{cll4}, it follow that
\begin{align}\label{al7}
\re\lj\frac1n\bigg(\rtr\bq_k^*\bd_k(z)\bq_k-\rtr\bd_k(z)\bigg)\rj^3\le C\(n^{-3/2}+n^{-1}\eta_n^2\)=o(1).
\end{align}
Combining with (\ref{al5}), (\ref{al6}), and (\ref{al7}), we get
\begin{align}\label{al9}
\re|\rtr\tilde\varepsilon_k(z)|^3=o(1)
\end{align}
uniformly in $z\in\mathcal{C}_n$. From $\varepsilon_k(z)=\tilde\varepsilon_k(z)-\big(m_n(z)-\re m_n(z)\big)\bi_2$, it suffices to show that
\begin{align*}
\re|m_n(z)-\re m_n(z)|^3=o(1)
\end{align*}
uniformly in $z\in\mathcal{C}_n$.

Let $\re_k(\cdot)$ denote the conditional expectation given $\{x_{jl},j,l>k\}$, then we have
\begin{align*}
&m_n(z)-\re m_n(z)=\frac1{2n}\sum_{k=1}^n\(\re_{k-1}\rtr\bd(z)-\re_{k}\rtr\bd(z)\)\\
&\qquad\qquad=\frac1{2n}\sum_{k=1}^n\(\re_{k-1}-\re_{k}\)\rtr\(\bd(z)-\bd_k(z)\)=\frac1{2n}\sum_{k=1}^n\gamma_k(z).
\end{align*}
Applying Lemma \ref{cll2} and Burkholder inequality (Lemma \ref{cll5}), it yields that
\begin{align*}
&\re\lj m_n(z)-\re m_n(z)\rj^3=\frac1{8n^3}\re\lj\sum_{k=1}^n \gamma_k(z)\rj^3\le Cn^{-3}\re\lj\sum_{k=1}^n \gamma_k(z)\rj^3\\
&\qquad\quad\le Cn^{-3}\re\(\sum_{k=1}^n\lj \gamma_k(z)\rj^2\)^{3/2}\le Cn^{-5/2}\sum_{k=1}^n\re\lj \gamma_k(z)\rj^3\\
&\qquad\quad\le Cn^{-5/2}\sum_{k=1}^n\re\lj \(\re_{k-1}-\re_k\)\rtr\(\bigg(\bi_2+\frac1n\bq_k^*\bd_k^2(z)\bq_k\bigg)\zeta_k(z)\)\rj^3\\
&\qquad\quad\le Cn^{-5/2}\sum_{k=1}^n\re\lj \(\re_{k-1}-\re_k\)\tilde t_n(z)\rtr\bigg(\bi_2+\frac1n\bq_k^*\bd_k^2(z)\bq_k\bigg)\rj^3\\
&\qquad\quad+ Cn^{-5/2}\sum_{k=1}^n\re\lj \(\re_{k-1}-\re_k\)\tilde t_n(z)\rtr\(\bigg(\bi_2+\frac1n\bq_k^*\bd_k^2(z)\bq_k\bigg)\zeta_k(z)\tilde\varepsilon_k(z)\)\rj^3.
\end{align*}
where the last inequality uses the fact $\zeta_k(z)=\tilde t_n(z)\bi_2+\tilde t_n(z)\zeta_k(z)\tilde\varepsilon_k(z)$. By (\ref{al10}) and Lemma \ref{cll4}, one gets
\begin{align*}
&\sum_{k=1}^n\re\lj \(\re_{k-1}-\re_k\)\tilde t_n(z)\rtr\bigg(\bi_2+\frac1n\bq_k^*\bd_k^2(z)\bq_k\bigg)\rj^3\\
=&\frac1{n^3}\sum_{k=1}^n\re\lj\re_{k-1}\tilde t_n(z)\bigg(\rtr\bq_k^*\bd_k^2(z)\bq_k-\rtr\bd_k^2(z)\bigg)\rj^3\\
\le&\frac C{n^2}\(n^{3/2}+\eta_n^2n^2\)=o(1).
\end{align*}
Using (\ref{al12}) and Lemma \ref{cll3}, we get
\begin{align*}
&\lj\rtr\(\bigg(\bi_2+\frac1n\bq_k^*\bd_k^2(z)\bq_k\bigg)^3\tilde\varepsilon_k^3(z)\)\rj
=\frac1{16}\lj\rtr\bigg(\bi_2+\frac1n\bq_k^*\bd_k^2(z)\bq_k\bigg)\rj^3\lj\rtr\tilde\varepsilon_k(z)\rj^3\\
\le& Cn^9\lj\rtr\bigg(\bi_2+\frac1n\bq_k^*\bd_k^2(z)\bq_k\bigg)\rj^3\le Cn^9\bigg(2+\frac{2n\eta_n^2n}{nv^2}\bigg)^3\le Cn^{18}.
\end{align*}
Employing (\ref{al10}), (\ref{al8}), and Lemma \ref{cll3}, we have
\begin{align*}
&\sum_{k=1}^n\re\lj \(\re_{k-1}-\re_k\)\tilde t_n(z)\rtr\(\bigg(\bi_2+\frac1n\bq_k^*\bd_k^2(z)\bq_k\bigg)\zeta_k(z)\tilde\varepsilon_k(z)\)\rj^3\\
\le&\sum_{k=1}^n\re\lj \rtr\(\bigg(\bi_2+\frac1n\bq_k^*\bd_k^2(z)\bq_k\bigg)\zeta_k(z)\tilde\varepsilon_k(z)\)\rj^3\\
=&4\sum_{k=1}^n\re\lj \rtr\(\bigg(\bi_2+\frac1n\bq_k^*\bd_k^2(z)\bq_k\bigg)^3\zeta_k^3(z)\tilde\varepsilon_k^3(z)\)\rj\\
\le&32\sum_{k=1}^n\re\lj \rtr\(\bigg(\bi_2+\frac1n\bq_k^*\bd_k^2(z)\bq_k\bigg)^3\tilde\varepsilon_k^3(z)\)\rj+o(n)\\
=&\sum_{k=1}^n\re\lj \rtr\bigg(\bi_2+\frac1n\bq_k^*\bd_k^2(z)\bq_k\bigg)\rtr\tilde\varepsilon_k(z)\rj^3+o(n)\\
\le&8\sum_{k=1}^n\re^{1/2}\lj \rtr\bigg(\bi_2+\frac1n\bq_k^*\bd_k^2(z)\bq_k\bigg)\rj^6\re^{1/2}\lj\rtr\tilde\varepsilon_k(z)\rj^6+o(n)=o(n).
\end{align*}
where the last inequality is from (\ref{al9}) and Lemma \ref{cll4}. Therefore, from the above inequalities we conclude that
\begin{align}\label{al15}
&\re\lj m_n(z)-\re m_n(z)\rj^3= o(n^{-5/2})+o(n^{-3/2})=o(1)
\end{align}
which completes the proof that
\begin{align*}
\mathcal{I}_3=o(1)
\end{align*}
uniformly in $z\in\mathcal{C}_n$.

Secondly, we find the approximation of $\re\rtr\varepsilon_k(z)$. Recall that
\begin{align*}
\re\rtr\varepsilon_k(z)=&\re\rtr\(\frac1{\sqrt n}x_{kk}-\frac1n\bq_k^*\bd_k(z)\bq_k\)+2\re m_n(z)\\
=&\frac1n\re\rtr\(\bd(z)-\bd_k(z)\)=-\frac1n\re\rtr\(\bigg(\bi_2+\frac1n\bq_k^*\bd_k^2(z)\bq_k\bigg)\zeta_k(z)\)\\
=&-\frac1n\re\rtr\zeta_k(z)-\frac1{n^2}\re\rtr\bigg(\bq_k^*\bd_k^2(z)\bq_k\bigg)\zeta_k(z)\\
=&-\frac1n\re\rtr\zeta_k(z)-\frac1{2n^2}\re\bigg(\rtr\(\bq_k^*\bd_k^2(z)\bq_k\)-\rtr\bd_k^2(z)\bigg)\rtr\zeta_k(z)-\frac1{2n^2}\re\rtr\bd_k^2(z)\rtr\zeta_k(z).
\end{align*}
Using (\ref{al11}) and Lemma \ref{cll4}, it follows that
\begin{align*}
&\frac1{2n^2}\re\lj\bigg(\rtr\(\bq_k^*\bd_k^2(z)\bq_k\)-\rtr\bd_k^2\bigg)\rtr\zeta_k(z)\rj\\
\le&\frac1{n^2v}\re^{1/2}\lj\rtr\(\bq_k^*\bd_k^2(z)\bq_k\)-\rtr\bd_k^2\rj^2\le\frac C{n^2v}\sqrt{n}=\frac C{\sqrt n \varepsilon_n}=o(1).
\end{align*}
Let $F_{nk}$ denote the ESD of $\bw_{nk}$, then by the interlacing theorem, we have
\begin{align*}
\|F_n-F_{nk}\|\le\frac1n.
\end{align*}
Combining the above inequality and $F_n\xrightarrow{a.s.}F$, one finds
\begin{align}\label{al13}
\max_{k\le n}\sup_x|F_{nk}(x)-F(x)|\to0,\quad a.s.
\end{align}
From the above inequality, it yields that
\begin{align}\label{al22}
\sup_{z\in\mathcal{C}_n}\re\lj\frac1{2n}\rtr\bd_k^2-m'(z)\rj
=&\sup_{z\in\mathcal{C}_n}\re\lj\frac1{2n}\rtr\bd_k^2-m'(z)\rj I(B_n^c)
+\sup_{z\in\mathcal{C}_n}\re\lj\frac1{2n}\rtr\bd_k^2-m'(z)\rj I(B_n)\\
\le&\sup_{z\in\mathcal{C}_n}\re\lj\frac1{2n}\rtr\bd_k^2-m'(z)\rj I(B_n^c)+\sup_{z\in\mathcal{C}_n}\frac2{v^2}{\rm P}(B_n)\notag\\
\le&\sup_{z\in\mathcal{C}_n}\re\lj\frac{n-1}n\int\frac{dF_{nk}}{(x-z)^2}-\int\frac{dF}{(x-z)^2}\rj I(B_n^c)+\sup_{z\in\mathcal{C}_n}\frac2{v^2}{\rm P}(B_n)\notag\\
\le&C\sup_x|F_{nk}(x)-F(x)|+o(n^{-t})=o(1).\notag
\end{align}
Therefore, we obtain that
\begin{align}\label{al14}
\re\rtr\varepsilon_k(z)=&-\frac1n\re\rtr\zeta_k(z)-\frac{m'(z)}{n}\re\rtr\zeta_k(z)+o(n^{-1})=-\frac{1+m'(z)}n\re\rtr\zeta_k(z)+o(n^{-1}),
\end{align}
which implies
\begin{align*}
\sum_{k=1}^n\re\rtr\varepsilon_k(z)=&-\frac{1+m'(z)}n\sum_{k=1}^n\re\rtr\zeta_k(z)+o(1)=2\(1+m'(z)\)\re m_n(z)+o(1),
\end{align*}
where $o(1)$ is uniform for $z\in\mathcal{C}_n$.

By (\ref{al13}), one has
\begin{align*}
\sup_{z\in\mathcal{C}_n}\lj\re m_n(z)-m(z)\rj=o(1)
\end{align*}
which implies
\begin{align}\label{al18}
\sum_{k=1}^n\re\rtr\varepsilon_k(z)=2\(1+m'(z)\)m(z)+o(1).
\end{align}

Finally, we investigate the limit of $\re\varepsilon_k^2(z)$. It is obvious that
\begin{align*}
\sum_{k=1}^n\re\rtr\varepsilon_k^2(z)=\frac12\sum_{k=1}^n\re\left[\rtr\varepsilon_k(z)-\re\rtr\varepsilon_k(z)\right]^2+\frac12\sum_{k=1}^n[\re\rtr\varepsilon_k(z)]^2.
\end{align*}
For the second term of the righthand side of the above equality, we get by (\ref{al8}) and (\ref{al14})
\begin{align*}
\sum_{k=1}^n\left|\re\rtr\varepsilon_k(z)\rj^2\le n\(Cn^{-1}\)^2\le Cn^{-1}=o(1).
\end{align*}
Note that
\begin{align*}
\rtr\varepsilon_k(z)-\re\rtr\varepsilon_k(z)=&\frac1{\sqrt n}\rtr x_{kk}-\frac1n\(\rtr\(\bq_k^*\bd_k(z)\bq_k\)-\re\rtr\bd_k(z)\)\\
=&\frac1{\sqrt n}\rtr x_{kk}-\frac1n\(\rtr\(\bq_k^*\bd_k(z)\bq_k\)-\rtr\bd_k(z)\)-\frac1n\(\rtr\bd_k(z)-\re\rtr\bd_k(z)\).
\end{align*}
It follows
\begin{align*}
&\re\(\rtr\varepsilon_k(z)-\re\rtr\varepsilon_k(z)\)^2=\frac{4\sigma^2}{n}+\frac1{n^2}\re\(\rtr\(\bq_k^*\bd_k(z)\bq_k\)-\rtr\bd_k(z)\)^2\\
&\qquad\qquad+\frac1{n^2}\re\(\rtr\bd_k(z)-\re\rtr\bd_k(z)\)^2.
\end{align*}
Examining the proof of (\ref{al15}), one can similarly prove that
\begin{align*}
\re\lj\frac1n\(\rtr\bd_k(z)-\re\rtr\bd_k(z)\)\rj^2=o(n^{-1}).
\end{align*}
Employing Lemma \ref{cll6}, it follows that
\begin{align*}
\re\(\rtr\(\bq_k^*\bd_k(z)\bq_k\)-\rtr\bd_k(z)\)^2=&\(M-\frac{3}{2}\)\re\sum_{j}\(\rtr \mathbbm{d}_{j,j}\)^2+\re\rtr \bd_k^2(z)+o(n)
\end{align*}
where $\mathbbm{d}_{j,j}$ the $2\times2$ diagonal block matrices of the matrix $\bd_k(z)$. Let $\bq_{kj}$ denote the quaternion column of $\bq_k$ with $j$th quaternion elements removed. Let $\bw_{nkj}$ be the matrix obtained from $\bw_{nk}$ with the $j$th quaternions column and row removed. Moreover, write
\begin{align*}
& \bd_{kj}(z)=(\bw_{nkj}-z\bi_{2n-4})^{-1}, \ \zeta_{kj}(z)= \((z +  m(z)) *\bi_2-\vk\)^{-1}\\
&\varepsilon_{kj}(z) = n^{-1/2}x_{jj}-n^{-1}\bq_{kj}^*\bd_{kj}(z)\bq_{kj} + m(z)*\bi_2.
\end{align*}
Using Lemma \ref{cll2}, we have
\begin{align*}
&\mathbbm{d}_{j,j}=-\zeta_{kj}(z)=-\frac1{z+m(z)}\bi_2-\frac1{z+m(z)}\zeta_{kj}(z)\varepsilon_{kj}(z)\\
&\qquad\qquad=m(z)\bi_2+m(z)\zeta_{kj}(z)\varepsilon_{kj}(z).
\end{align*}
It can be verified that (\ref{al8}) holds for $\zeta_{kj}(z)$. Thus, combining with (\ref{al6}), Lemma \ref{cll4} and Theorem 1.1 in \cite{yin2014rate}, we get
\begin{align}\label{al23}
&\re\lj\rtr\mathbbm{d}_{j,j}-2m(z)\rj^2\le 4\re\lj\varepsilon_{kj}(z)\rj^2+o(1)\\
\le&\frac Cn\re\left\|x_{jj}\right\|_Q^2+\frac C{n^{2}}\re\lj\rtr\(\bq_{kj}^*\bd_{kj}(z)\bq_{kj}\)-\rtr\(\bd_{kj}(z)\)\rj^2\notag\\
&+\frac C{n^2}\re\lj\rtr\bd_{kj}(z)-\rtr\bd(z)\rj^2+C\re\lj m_n(z)-m(z)\rj^2+o(1)\notag\\
\le&2Cn^{-1}+Cn^{-2}+Cn^{-4/5}+o(1)=o(1).\notag
\end{align}
Therefore, one obtains that
\begin{align}\label{al19}
\sum_{k=1}^n\re\rtr\varepsilon_k^2(z)=2{\sigma^2}+\(2M-{3}\)m^2(z)+m'(z)+o(1).
\end{align}

Then from (\ref{al17}), (\ref{al18}), and (\ref{al19}), we get
\begin{align}\label{al21}
&n\delta_n(z)  =-{t_n^2(z)}\(1+m'(z)\)m(z)\\
&\qquad-{t_n^3(z)}\({\sigma^2}+{\(M-\frac{3}2\)}m^2(z)+\frac12m'(z)\)+o(1)\notag
\end{align}
which implies
\begin{align*}
\delta_n(z) =o(1).
\end{align*}
Note that
\begin{align*}
&\re m_n(z)=-\frac1{2n}\sum_{k=1}^n\re\rtr\zeta_k(z)\\
=&-\frac1{2n}\sum_{k=1}^n\re\rtr\(t_n(z)\bi_2+ t_n(z)\zeta_k(z)\varepsilon_k(z)\)
=-t_n(z)+\delta_n(z).
\end{align*}
By the above equality and (\ref{al16}), it yields
\begin{align}\label{al20}
t_n(z)=-m\(z+\delta_n(z)\)=-m(z)+o(1).
\end{align}
Substituting (\ref{al20}) into (\ref{al21}), we conclude
\begin{align*}
&n\delta_n(z)  ={m^3(z)}\({\sigma^2}-1-\frac12m'(z)+\(M-\frac{3}2\)m^2(z)\)+o(1)\notag.
\end{align*}
Hence,
\begin{align*}
\re M_n(z)=\(1+m'(z)\){m^3(z)}\({\sigma^2}-1-\frac12m'(z)+\(M-\frac{3}2\)m^2(z)\)+o(1).
\end{align*}

\subsection{Convergence of the process $M_n(z)-\re M_n(z)$}

In this section, we establish the convergence of the process $M_n(z)-\re M_n(z)$. For this aim, we proceed in our proof by taking several steps.

\subsubsection{Finite dimensional convergence of $M_n(z)-\re M_n(z)$}
It is obvious from Lemma \ref{cll3} that
\begin{align}\label{eq2}
&M_n(z)-\re M_n(z)\\
=&\frac12\sum_{k=1}^n\gamma_k=-\frac12\sum_{k=1}^n\(\re_{k-1}-\re_k\)\rtr\(\bigg(\bi_2+\frac1n\bq_k^*\bd_k^2(z)\bq_k\bigg)\zeta_k(z)\)\notag\\
=&-\frac14\sum_{k=1}^n\(\re_{k-1}-\re_k\)\rtr\bigg(\bi_2+\frac1n\bq_k^*\bd_k^2(z)\bq_k\bigg)\rtr\zeta_k(z)\notag\\
=&-\frac12\sum_{k=1}^n\(\re_{k-1}-\re_k\)b_k(z)\rtr\bigg(\bi_2+\frac1n\bq_k^*\bd_k^2(z)\bq_k\bigg)\notag\\
&-\frac18\sum_{k=1}^n\(\re_{k-1}-\re_k\)b_k(z) g_k(z)\rtr\bigg(\bi_2+\frac1n\bq_k^*\bd_k^2(z)\bq_k\bigg)\rtr\zeta_k(z)\notag\\
=&-\frac1{2n}\sum_{k=1}^n\re_{k-1}b_k(z)\(\rtr\bigg(\bq_k^*\bd_k^2(z)\bq_k\bigg)-\rtr\bd_k^2(z)\)\notag\\
&-\frac18\sum_{k=1}^n\(\re_{k-1}-\re_k\)b_k(z) g_k(z)\rtr\bigg(\bi_2+\frac1n\bq_k^*\bd_k^2(z)\bq_k\bigg)\rtr\zeta_k(z)\notag\\
\triangleq&-\frac1{2n}\sum_{k=1}^n\re_{k-1}b_k(z)\(\rtr\bigg(\bq_k^*\bd_k^2(z)\bq_k\bigg)-\rtr\bd_k^2(z)\)-\frac18\sum_{k=1}^n\(\re_{k-1}-\re_k\)a_k\notag.
\end{align}
where
\begin{align*}
b_k(z)&=\frac1{z+\frac1{2n}\rtr\bd_k(z)},\quad
g_k(z)=\frac1{\sqrt n}\rtr x_{kk}-\frac1n\rtr\(\bq_k^*\bd_k(z)\bd_k\)+\frac1{n}\rtr\bd_k(z).
\end{align*}
Notice that
\begin{align*}
a_k =&b_k(z) g_k(z)\rtr\bigg(\bi_2+\frac1n\bq_k^*\bd_k^2(z)\bq_k\bigg)\(2b_k+\frac12 b_k g_k\rtr\zeta_k(z)\)\\
=&2b_k^2(z) g_k(z)\rtr\bigg(\bi_2+\frac1n\bq_k^*\bd_k^2(z)\bq_k\bigg)+\frac12 b_k^2(z) g_k^2(z)\rtr\bigg(\bi_2+\frac1n\bq_k^*\bd_k^2(z)\bq_k\bigg)\rtr\zeta_k(z)\\
=&4b_k^2(z) g_k(z)\bigg(1+\frac1{2n}\rtr\bd_k^2(z)\bigg)+\frac2n b_k^2 (z)g_k(z)\rtr\(\bq_k^*\bd_k^2(z)\bq_k-\rtr\bd_k^2(z)\)\\
&+ b_k^2(z) g_k^2(z)\rtr\(\bigg(\bi_2+\frac1n\bq_k^*\bd_k^2(z)\bq_k\bigg)\zeta_k(z)\)\\
\triangleq&a_{k1}+a_{k2}+a_{k3}.
\end{align*}
Comparing with (\ref{al10}), it can be verified
\begin{align}\label{eq1}
\lj b_k(z)\rj<1.
\end{align}
Employing (\ref{al8}) and Lemma \ref{cll4}, we get
\begin{align*}
&\re\lj\sum_{k=1}^n\(\re_{k-1}-\re_k\)a_{k2}\rj^2=\sum_{k=1}^n\re\lj\(\re_{k-1}-\re_k\)a_{k2}\rj^2\\
\le&\frac4{n^2}\sum_{k=1}^n\re\lj g_k(z)\rtr\(\bq_k^*\bd_k^2(z)\bq_k-\rtr\bd_k^2(z)\)\rj^2\\
\le&\frac4{n^2}\sum_{k=1}^n\re^{1/2}\lj g_k(z)\rj^4\re^{1/2}\lj\rtr\(\bq_k^*\bd_k^2(z)\bq_k-\rtr\bd_k^2(z)\)\rj^4\\
\le&\frac C{n}\(\frac1n+\frac{\eta_n^2}{n^{1/2}}\)\(n+n^{3/2}\eta_n^2\)=o(1)
\end{align*}
and
\begin{align*}
&\re\lj\sum_{k=1}^n\(\re_{k-1}-\re_k\)a_{k3}\rj^2=\sum_{k=1}^n\re\lj\(\re_{k-1}-\re_k\)a_{k3}\rj^2\\
\le&\sum_{k=1}^n\re\lj g_k^2(z)\rtr\(\bigg(\bi_2+\frac1n\bq_k^*\bd_k^2(z)\bq_k\bigg)\zeta_k(z)\)\rj^2\\
\le&4\sum_{k=1}^n\re\lj g_k^2(z)\rtr\(\bigg(\bi_2+\frac1n\bq_k^*\bd_k^2(z)\bq_k\bigg)\)\rj^2\\
\le&C\sum_{k=1}^n\re\lj g_k^2(z)\bigg(2+\frac1n\rtr\bd_k^2(z)\bigg)\rj^2+\frac C{n^2}\sum_{k=1}^n\re\lj g_k^2(z)\(\rtr\bigg(\bq_k^*\bd_k^2(z)\bq_k\bigg)-\rtr\bd_k^2(z)\)\rj^2\\
\le&C\sum_{k=1}^n\re\lj 2+\frac1n\rtr\bd_k^2(z)\rj^2\(\frac1{n^2}\|\bd_k(z)\|^4+\frac{\eta_n^4}n\|\bd_k(z)\|^4\)+o(1)=o(1).
\end{align*}
where $o(1)$ is uniform in $z\in\mathcal{C}_n$. Hence, it follow that
\begin{align*}
M_n(z)-\re M_n(z)=&-\frac1{2n}\sum_{k=1}^n\re_{k-1}b_k(z)\(\rtr\bigg(\bq_k^*\bd_k^2(z)\bq_k\bigg)-\rtr\bd_k^2(z)\)\\
&-\frac12\sum_{k=1}^n\re_{k-1}b_k^2(z) g_k(z)\bigg(1+\frac1{2n}\rtr\bd_k^2(z)\bigg)+o_{L_2}(1)\\
\triangleq&\sum_{k=1}^n\re_{k-1}\psi_k(z)+o_{L_2}(1)=\sum_{k=1}^n\re_{k-1}\frac d{dz}\(\frac12b_k(z)g_k(z)\)+o_{L_2}(1)\\
\triangleq&\sum_{k=1}^n\re_{k-1}\frac d{dz}\phi_k(z)+o_{L_2}(1).
\end{align*}
Here, $o_{L_2}(1)$ is uniform for $z\in\mathcal{C}_n$ in the sense of $L_2$ convergence.

Now, we return assume $z\in\mathbb{C}_0$. Let $\{z_t,t=1,\cdots,m\}$ be $m$ different points belongs to $\mathbb{C}_0$. Then we only need to deducing the weak convergence of the vector martingale
\begin{align*}
\mb Z_n=\sum_{k=1}^n\re_{k-1}\(\psi_{k}(z_1),\cdots,\psi_k(z_m)\)=\sum_{k=1}^n\re_{k-1}\boldsymbol\Psi_k.
\end{align*}

Let
\begin{align*}
\Gamma_n(z_j,z_l)=\sum_{k=1}^n\re_k\left[\re_{k-1}\frac d{dz_j}\phi_k(z_j)\re_{k-1}\frac d{dz_l}\phi_k(z_l)\right].
\end{align*}
Using Lemma \ref{cll7}, it suffices to show that Lyapounov's condition holds and $\Gamma_n$ converges in probability.

At first, applying Lemma \ref{cll4}, we have
\begin{align*}
&\sum_{k=1}^n\re\|\re_{k-1}\boldsymbol\Psi_k\|^4=\sum_{k=1}^n\re\(\sum_{j=1}^m|\re_{k-1}\psi_k(z_j)|^2\)^2
\le m\sum_{k=1}^n\sum_{j=1}^m\re|\re_{k-1}\psi_k(z_j)|^4\\
\le&\frac C{n^4}\sum_{k=1}^n\sum_{j=1}^m\re\lj\rtr\bigg(\bq_k^*\bd_k^2(z)\bq_k\bigg)-\rtr\bd_k^2(z)\rj^4+C\sum_{k=1}^n\sum_{j=1}^m\re\lj\(1+\frac1{2n}\rtr\bd_k^{2}(z_j)\)g_k(z)\rj^4\\
\le&C\(n^{-1}+\eta_n^4\)+C\sum_{k=1}^n\sum_{j=1}^m\re\lj g_k(z)\rj^4\\
\le& o\(1\)+C\sum_{k=1}^n\sum_{j=1}^m\re\lj\frac1{\sqrt n}\rtr x_{kk}\rj^4+\frac C{n^4}\sum_{k=1}^n\sum_{j=1}^m\re\lj \rtr\(\bq_k^*\bd_k(z)\bd_k\)-\rtr\bd_k(z)\rj^4=o(1).
\end{align*}

We are in a position to derive the limit of $\Gamma_n$. For any $z_1,z_2\in\mathbb{C}_0$, employing Vitali's lemma, our goal transform into finding the limit of
\begin{align*}
\sum_{k=1}^n\re_k\left[\re_{k-1}\phi_k(z_1)\re_{k-1}\phi_k(z_2)\right]=\frac14\sum_{k=1}^n\re_k\left[\re_{k-1}b_k(z_1)g_k(z_1)\re_{k-1}b_k(z_2)g_k(z_2)\right].
\end{align*}
From (\ref{al22}) and Lemma \ref{cll6}, it follows
\begin{align*}
&\sum_{k=1}^n\re_k\left[\re_{k-1}\phi_k(z_1)\re_{k-1}\phi_k(z_2)\right]\\
=&\frac14m(z_1)m(z_2)\sum_{k=1}^n\re_k\left[\re_{k-1}g_k(z_1)\re_{k-1}g_k(z_2)\right]+o_{L_2}(1)\\
=&m(z_1)m(z_2)\sigma^2+\frac1{4n^2}m(z_1)m(z_2)\sum_{k=1}^n\re_k\bigg[\re_{k-1}\(\rtr\(\bq_k^*\bd_k(z_1)\bd_k\)-\rtr\bd_k(z_1)\)\\
&\quad\quad\times\re_{k-1}\(\rtr\(\bq_k^*\bd_k(z_2)\bd_k\)-\rtr\bd_k(z_2)\)\bigg]+o_{L_2}(1)\\
=&m(z_1)m(z_2)\sigma^2+\frac{2M-{3}}{8n^2}m(z_1)m(z_2)\sum_{k=1}^n\sum_{j>k}\rtr \mathbbm{d}_{jj1}\rtr \mathbbm{d}_{jj2}\\
&+\frac1{4n^2}m(z_1)m(z_2)\sum_{k=1}^n\sum_{j,l>k}\rtr \widetilde{\mathbbm{d}}_{jl1} \widetilde{\mathbbm{d}}_{lj2}+o_{L_2}(1)\\
=&m(z_1)m(z_2)\sigma^2+\frac{2M-{3}}{4}m^2(z_1)m^2(z_2)+\frac1{4n^2}m(z_1)m(z_2)\sum_{k=1}^n\sum_{j,l>k}\rtr \widetilde{\mathbbm{d}}_{jl1}\widetilde{\mathbbm{d}}_{lj2}+o_{L_2}(1)\\
=&m(z_1)m(z_2)\sigma^2+\frac{2M-{3}}{4}m^2(z_1)m^2(z_2)+\frac1{4n}m(z_1)m(z_2)\sum_{k=1}^nS_k+o_{L_2}(1).
\end{align*}
where $\widetilde{\mathbbm{d}}_{jjl}$ is the $2\times2$ diagonal block matrices of the matrix $\re_{k-1}\bd_k(z_l),l=1,2$ and the last second equality is from (\ref{al23}).

Now, we only need to find the limit of $S_1$. Let  $\be_j \ (j=1,\cdots,k-1,k+1,\cdots,n)$ be the $(2 n-2)\times2$ matrix whose $j$th (or $(j-1)$th) element is $\bi_2$ and others are ${\bf 0}$ if $j<k$ (or $j>k$ correspondingly). Recall that
\begin{align*}
\bw_{nk}=\frac1{\sqrt n}\sum_{j\neq l}\be_jx_{jl}\be_l'
\end{align*}
Multiplying both sides by $\bd_k(z)$, we get
\begin{align}\label{al24}
\bi_{2n-2}+z\bd_k(z)=\frac1{\sqrt n}\sum_{j\neq l}\be_jx_{jl}\be_l'\bd_k(z)
\end{align}
For $j,l\neq k$, define
\begin{align*}
&\bw_{kjl}=\bw_{nk}-\frac1{\sqrt n}\delta_{jl}\(\be_jx_{jl}\be_l'+\be_lx_{lj}\be_j'\),\ \bd_{kjl}(z)=\(\bw_{kjl}-z\bi_{2n-2}\)^{-1},
\end{align*}
and
\begin{align*}
\delta_{jl}=\begin{cases}
1,&\quad{\rm if} \ j\neq l,\\
1/2,&\quad{\rm if} \ j=l.
\end{cases}
\end{align*}
Using the identity
\begin{align*}
\ba^{-1}-\bb^{-1}=\bb^{-1}\(\bb-\ba\)\ba^{-1},
\end{align*}
it is obvious
\begin{align}\label{al25}
\bd_k(z)-\bd_{kjl}(z)=-\frac1{\sqrt n}\delta_{jl}\bd_{kjl}(z)\(\be_jx_{jl}\be_l'+\be_lx_{lj}\be_j'\)\bd_{k}(z).
\end{align}
By (\ref{al24}) and (\ref{al25}), we obtain
\begin{align*}
&z\re_{k-1}\bd_k(z)=-\bi_{2n-2}+\frac1{\sqrt n}\sum_{j,l>k}\be_jx_{jl}\be_l'\re_{k-1}\bd_{kjl}(z)\\
&\quad\quad-\frac1{ n}\sum_{j, l\neq k}\delta_{jl}\re_{k-1}\(\be_jx_{jl}\be_l'\bd_{kjl}(z)\(\be_jx_{jl}\be_l'+\be_lx_{lj}\be_j'\)\bd_{k}(z)\)\\
=&-\bi_{2n-2}+\frac1{\sqrt n}\sum_{j,l>k}\be_jx_{jl}\be_l'\re_{k-1}\bd_{kjl}(z)\\
&-\frac1{ 2n}\sum_{j, l\neq k}\delta_{jl}\re_{k-1}\(\|x_{jl}\|_Q^2\rtr\(\be_l'\bd_{kjl}(z)\be_l\)\be_j\be_j'\bd_{k}(z)\)\\
&-\frac1{ n}\sum_{j, l\neq k}\delta_{jl}\re_{k-1}\(\be_jx_{jl}\be_l'\bd_{kjl}(z)\be_jx_{jl}\be_l'\bd_{k}(z)\)\\
=&-\bi_{2n-2}+\frac1{\sqrt n}\sum_{j,l>k}\be_jx_{jl}\be_l'\re_{k-1}\bd_{kjl}(z)-\frac{\(n-3/2\)m(z)}{ n}\sum_{j\neq k}\be_j\be_j'\re_{k-1}\bd_{k}(z)\\
&-\frac{m(z)}{ n}\sum_{j, l\neq k}\delta_{jl}\re_{k-1}\(\(\|x_{jl}\|_Q^2-1\)\be_j\be_j'\bd_{k}(z)\)\\
&-\frac1{ 2n}\sum_{j, l\neq k}\delta_{jl}\re_{k-1}\(\|x_{jl}\|_Q^2\(\rtr\(\be_l'\bd_{kjl}(z)\be_l\)-2m(z)\)\be_j\be_j'\bd_{k}(z)\)\\
&-\frac1{ n}\sum_{j, l\neq k}\delta_{jl}\re_{k-1}\(\be_jx_{jl}\be_l'\bd_{kjl}(z)\be_jx_{jl}\be_l'\bd_{k}(z)\)\\
\triangleq&-\bi_{2n-2}+\ba_{k1}(z)+\ba_{k2}(z)+\ba_{k3}(z)+\ba_{k4}(z)+\ba_{k5}(z)
\end{align*}
which implies that
\begin{align*}
z_1S_k=&-\frac1{n}\sum_{i_1>k}\rtr\(\be_{i_1}'\re_{k-1}\bd_k(z_2)\be_{i_1}\)+\frac1{n}\sum_{i_1,i_2>k}\rtr\(\be_{i_1}'\ba_{k1}(z_1)\be_{i_2}\be_{i_2}'\re_{k-1}\bd_k(z_2)\be_{i_1}\)\\
&+\frac1{n}\sum_{i_1,i_2>k}\rtr\(\be_{i_1}'\ba_{k2}(z_1)\be_{i_2}\be_{i_2}'\re_{k-1}\bd_k(z_2)\be_{i_1}\)\\
&+\frac1{n}\sum_{i_1,i_2>k}\rtr\(\be_{i_1}'\ba_{k3}(z_1)\be_{i_2}\be_{i_2}'\re_{k-1}\bd_k(z_2)\be_{i_1}\)\\
&+\frac1{n}\sum_{i_1,i_2>k}\rtr\(\be_{i_1}'\ba_{k4}(z_1)\be_{i_2}\be_{i_2}'\re_{k-1}\bd_k(z_2)\be_{i_1}\)\\
&+\frac1{n}\sum_{i_1,i_2>k}\rtr\(\be_{i_1}'\ba_{k5}(z_1)\be_{i_2}\be_{i_2}'\re_{k-1}\bd_k(z_2)\be_{i_1}\).
\end{align*}

We assert that the last three terms of the above equality are negligible. It can be verified that
\begin{align}\label{al26}
\lj\sum_{i_2>k}\rtr\(\be_{i_1}'\bd_{k}(z_1)\be_{i_2}\be_{i_2}'\re_{k-1}\bd_k(z_2)\be_{i_1}\)\rj\le2v_0^{-2}.
\end{align}
Using (\ref{al26}) and Cauchy-Schwarz inequality, one finds
\begin{align*}
&\re\lj\sum_{i_2>k}\rtr\(\be_{i_1}'\ba_{k3}(z_1)\be_{i_2}\be_{i_2}'\re_{k-1}\bd_k(z_2)\be_{i_1}\)\rj\\
=&\frac{|m(z_1)|}{ n}\re\lj\re_{k-1}\(\sum_{l\neq k}\delta_{i_1l}\(\|x_{{i_1}l}\|_Q^2-1\)\sum_{i_2>k}\rtr\(\be_{i_1}'\bd_{k}(z_1)\be_{i_2}\be_{i_2}'\re_{k-1}\bd_k(z_2)\be_{i_1}\)\)\rj\\
\le&\frac{1}{ n}\re^{1/2}\lj\sum_{l\neq k}\delta_{i_1l}\(\|x_{{i_1}l}\|_Q^2-1\)\rj^2\re^{1/2}\lj\sum_{i_2>k}\rtr\(\be_{i_1}'\bd_{k}(z_1)\be_{i_2}\be_{i_2}'\re_{k-1}\bd_k(z_2)\be_{i_1}\)\rj^2\\
\le&\frac{2}{ nv_0}\(\sum_{l\neq k}\re\lj\|x_{{i_1}l}\|_Q^2-1\rj^2\)^{1/2}=O(n^{-1/2}).
\end{align*}
By definition, we have
\begin{align*}
&\re\lj\rtr\be_l'\(\bd_{k}(z_1)-\bd_{ki_1l}(z_1)\)\be_l\rj^2\\
=&\frac1{ n}\delta_{jl}\re\lj\rtr\be_l'\bd_{ki_1l}(z_1)\(\be_jx_{jl}\be_l'+\be_lx_{lj}\be_j'\)\bd_{k}(z_1)\be_l\rj^2\\
\le&\frac2{ n}\re\left\|\be_l'\bd_{ki_1l}(z_1)\(\be_jx_{jl}\be_l'+\be_lx_{lj}\be_j'\)\bd_{k}(z_1)\be_l\right\|^2\\
\le&2{\eta_n^2}\re\left\|\bd_{ki_1l}(z_1)\bd_{k}(z_1)\right\|^2\le2{\eta_n^2}v_0^{-4}=o(1).
\end{align*}
From(\ref{al23}) and the above inequality, it follow that
\begin{align}\label{al27}
&\re\lj\rtr\(\bd_{ki_1l}(z_1)\be_l\)-2m(z)\rj^2=o(1).
\end{align}
Then, we get by (\ref{al26}) and (\ref{al27})
\begin{align*}
&\re\lj\sum_{i_2>k}\rtr\(\be_{i_1}'\ba_{k4}(z_1)\be_{i_2}\be_{i_2}'\re_{k-1}\bd_k(z_2)\be_{i_1}\)\rj\\
=&\frac1{ 2n}\re\lj\re_{k-1}\(\sum_{ l\neq k}\delta_{i_1l}\|x_{i_1l}\|_Q^2\(\rtr\(\be_l'\bd_{ki_1l}(z_1)\be_l\)-2m(z_1)\)\sum_{i_2>k}\rtr\(\be_{i_1}'\bd_{k}(z_1)\be_{i_2}\be_{i_2}'\re_{k-1}\bd_k(z_2)\be_{i_1}\)\)\rj
\\
\le&\frac1{ nv_0^2}\re\lj\sum_{ l\neq k}\delta_{i_1l}\|x_{i_1l}\|_Q^2\(\rtr\(\be_l'\bd_{ki_1l}(z_1)\be_l\)-2m(z_1)\)\rj\\
\le&\frac1{ nv_0^2}\sum_{ l\neq k}\re\lj\rtr\(\be_l'\bd_{ki_1l}(z_1)\be_l\)-2m(z_1)\rj=o(1).
\end{align*}
Furthermore, we find
\begin{align*}
&\frac1n\re\lj\sum_{i_1,i_2>k}\rtr\(\be_{i_1}'\ba_{k5}(z_1)\be_{i_2}\be_{i_2}'\re_{k-1}\bd_k(z_2)\be_{i_1}\)\rj\\
=&\frac{\delta_{i_1l}}{ n^2}\re\lj\sum_{i_1>k}\sum_{l\neq k}\re_{k-1}\rtr\(x_{i_1l}\be_l'\bd_{ki_1l}(z_1)\be_{i_1}x_{i_1l}\sum_{i_2>k}\be_l'\bd_{k}(z_1)\be_{i_2}\be_{i_2}'\re_{k-1}\bd_k(z_2)\be_{i_1}\)\rj\\
\le&\frac{C}{ n^2v_0}\re\(\sum_{i_1>k}\sum_{l\neq k}\left\|x_{i_1l}\right\|_Q^2\left\|\sum_{i_2>k}\be_l'\bd_{k}(z_1)\be_{i_2}\be_{i_2}'\re_{k-1}\bd_k(z_2)\be_{i_1}\right\|\)\\
\le&\frac{C}{ n^2v_0}\(\sum_{i_1>k}\sum_{l\neq k}\re\left\|x_{i_1l}\right\|_Q^4\)^{1/2}
\(\sum_{i_1>k}\sum_{l\neq k}\re\left\|\sum_{i_2>k}\be_l'\bd_{k}(z_1)\be_{i_2}\be_{i_2}'\re_{k-1}\bd_k(z_2)\be_{i_1}\right\|^2\)^{1/2}\\
\le&\frac{C}{ n}
\(\sum_{i_1>k}\sum_{l\neq k}\re\left\|\sum_{i_2>k}\be_l'\bd_{k}(z_1)\be_{i_2}\be_{i_2}'\re_{k-1}\bd_k(z_2)\be_{i_1}\right\|^2\)^{1/2}= O(n^{-1/2}).
\end{align*}
Hence,
\begin{align*}
z_1S_k=&-\frac1{n}\sum_{i_1>k}\rtr\(\be_{i_1}'\re_{k-1}\bd_k(z_2)\be_{i_1}\)+\frac1{n}\sum_{i_1,i_2>k}\rtr\(\be_{i_1}'\ba_{k1}(z_1)\be_{i_2}\be_{i_2}'\re_{k-1}\bd_k(z_2)\be_{i_1}\)\\
&+\frac1{n}\sum_{i_1,i_2>k}\rtr\(\be_{i_1}'\ba_{k2}(z_1)\be_{i_2}\be_{i_2}'\re_{k-1}\bd_k(z_2)\be_{i_1}\)+o_{L_2}(1)\\
=&-2{m(z_2)}\(1-\frac kn\)+\frac1{n}\sum_{i_1,i_2>k}\rtr\(\be_{i_1}'\ba_{k1}(z_1)\be_{i_2}\be_{i_2}'\re_{k-1}\bd_k(z_2)\be_{i_1}\)\\
&-\frac{m(z_1)}{n}\sum_{i_1,i_2>k}\rtr\(\be_{i_1}'\re_{k-1}\bd_{k}(z_1)\be_{i_2}\be_{i_2}'\re_{k-1}\bd_k(z_2)\be_{i_1}\)+o_{L_2}(1)
\end{align*}
where the last inequality is from (\ref{al23}).

Now, let us evaluate the contributive components in the expression of $S_1$. By(\ref{al25}),  we have
\begin{align*}
&\sum_{i_2>k}\rtr\(\be_{i_1}'\ba_{k1}(z_1)\be_{i_2}\be_{i_2}'\re_{k-1}\bd_k(z_2)\be_{i_1}\)\\
=&\frac1{\sqrt n}\sum_{i_2>k}\sum_{l>k}\rtr\(x_{i_1l}\be_l'\re_{k-1}\bd_{ki_1l}(z_1)\be_{i_2}\be_{i_2}'\re_{k-1}\bd_k(z_2)\be_{i_1}\)\\
=&-\frac{m(z_2)}{ n}\sum_{l ,i_2>k}\delta_{i_1l}\rtr\(\be_l'\re_{k-1}\bd_{ki_1l}(z_1)\be_{i_2}\be_{i_2}'\re_{k-1}\bd_{ki_1l}(z_2)
\be_l\)\\
&-\frac{m(z_2)}{ n}\sum_{l ,i_2>k}\delta_{i_1l}\(\|x_{i_1l}\|_Q^2-1\)\rtr\(\be_l'\re_{k-1}\bd_{ki_1l}(z_1)\be_{i_2}\be_{i_2}'\re_{k-1}\bd_{ki_1l}(z_2)
\be_l\)\\
&-\frac1{ 2n}\sum_{l ,i_2>k}\delta_{i_1l}\|x_{i_1l}\|_Q^2\(\rtr\(\be_{i_1}'\bd_{k}(z_2)\be_{i_1}\)-2m(z_2)\)\rtr\(\be_l'\re_{k-1}\bd_{ki_1l}(z_1)\be_{i_2}\be_{i_2}'\re_{k-1}\bd_{ki_1l}(z_2)
\be_l\)\\
&-\frac1{ n}\sum_{i_2>k}\sum_{l>k}\delta_{i_1l}\rtr\(x_{i_1l}\be_l'\re_{k-1}\bd_{ki_1l}(z_1)\be_{i_2}\be_{i_2}'\re_{k-1}\bigg(\bd_{ki_1l}(z_2)
\be_{i_1}x_{i_1l}\be_l'\bd_{k}(z_2)\bigg)\be_{i_1}\)\\
&+\frac1{\sqrt n}\sum_{i_2>k}\sum_{l>k}\rtr\(x_{i_1l}\be_l'\re_{k-1}\bd_{ki_1l}(z_1)\be_{i_2}\be_{i_2}'\re_{k-1}\bd_{ki_1l}(z_2)\be_{i_1}\)\\
\triangleq&-\frac{m(z_2)}{ n}\sum_{l ,i_2>k}\delta_{i_1l}\rtr\(\be_l'\re_{k-1}\bd_{ki_1l}(z_1)\be_{i_2}\be_{i_2}'\re_{k-1}\bd_{ki_1l}(z_2)
\be_l\)+a_{ki_11}+a_{ki_12}+a_{ki_13}+a_{ki_14}.
\end{align*}
It is obvious from (\ref{al23}) and (\ref{al25}) that
\begin{align*}
\re|a_{ki_11}|^2=&\frac{m^2(z_2)}{ n^2}\re\lj\sum_{l ,i_2>k}\delta_{i_1l}\(\|x_{i_1l}\|_Q^2-1\)\rtr\(\be_l'\re_{k-1}\bd_{ki_1l}(z_1)\be_{i_2}\be_{i_2}'\re_{k-1}\bd_{ki_1l}(z_2)
\be_l\)\rj^2\\
=&\frac{m^2(z_2)}{ n^2}\sum_{l>k}\re\(\delta_{i_1l}^2\(\|x_{i_1l}\|_Q^2-1\)^2\lj\rtr\(\sum_{i_2>k}\be_l'\re_{k-1}\bd_{ki_1l}(z_1)\be_{i_2}\be_{i_2}'\re_{k-1}\bd_{ki_1l}(z_2)
\be_l\)^2\rj\)\\
\le&\frac{Cm^2(z_2)}{ n^2}\sum_{l>k}\re\(\|x_{i_1l}\|_Q^2-1\)^2=O(n^{-1})
\end{align*}
and
\begin{align*}
\re|a_{ki_12}|\le&\frac C{ n}\sum_{l >k}\re\lj\|x_{i_1l}\|_Q^2\(\rtr\(\be_{i_1}'\bd_{k}(z_2)\be_{i_1}\)-2m(z_2)\)\rj\\
\le&\frac C{ n}\sum_{l >k}\re^{1/2}\|x_{i_1l}\|_Q^4\re^{1/2}\lj\rtr\(\be_{i_1}'\bd_{k}(z_2)\be_{i_1}\)-2m(z_2)\rj^2=o(1).
\end{align*}
Using (\ref{al25}) and (\ref{al26}), we get
\begin{align*}
\re|a_{ki_13}|
\le&\frac C{ n}\(\sum_{l>k}\re\left\|x_{i_1l}\right\|_Q^4\left\|\sum_{i_2>k}\be_l'\re_{k-1}\bd_{ki_1l}(z_1)\be_{i_2}\be_{i_2}'\bd_{ki_1l}(z_2)
\be_{i_1}\right\|^2\)^{1/2}\\
&\quad\quad\times\(\sum_{l>k}\re\left\|\be_l'\bd_{k}(z_2)\be_{i_1}\right\|^2\)^{1/2}\\
\le&\frac C{ n}\(\sum_{l>k}\re\left\|\sum_{i_2>k}\be_l'\re_{k-1}\bd_{ki_1l}(z_1)\be_{i_2}\be_{i_2}'\bd_{ki_1l}(z_2)
\be_{i_1}\right\|^2\)^{1/2}\\
\le&\frac C{ n}\(\sum_{l>k}\re\left\|\sum_{i_2>k}\be_l'\re_{k-1}\bd_{k}(z_1)\be_{i_2}\be_{i_2}'\bd_{k}(z_2)
\be_{i_1}\right\|^2\)^{1/2}+o(1)=o(1).
\end{align*}
and
\begin{align*}
\re|a_{ki_14}|^2\le&\frac C{ n}\sum_{l>k}\re\|x_{i_1l}\|_Q^2\left\|\sum_{i_2>k}\be_l'\re_{k-1}\bd_{ki_1l}(z_1)\be_{i_2}\be_{i_2}'\re_{k-1}\bd_{ki_1l}(z_2)\be_{i_1}\right\|^2\\
&+\frac C{ n}\Bigg|\sum_{l_1\neq l_2>k}\re x_{i_1l_1}x_{l_2 i_1}\rtr\(\sum_{i_2>k}\be_{l_1}'\re_{k-1}\bd_{ki_1l_1}(z_1)\be_{i_2}\be_{i_2}'\re_{k-1}\bd_{ki_1l_1}(z_2)\be_{i_1}\)\\
&\quad\quad\times\rtr\(\sum_{i_2>k}\be_{l_2}'\re_{k-1}\bd_{ki_1l_2}(\bar z_1)\be_{i_2}\be_{i_2}'\re_{k-1}\bd_{ki_1l_2}(\bar z_2)\be_{i_1}\)\Bigg|\\
\le&\frac C{ n}\sum_{l>k}\re\left\|\sum_{i_2>k}\be_l'\re_{k-1}\bd_{k}(z_1)\be_{i_2}\be_{i_2}'\re_{k-1}\bd_{k}(z_2)\be_{i_1}\right\|^2\\
&+\frac C{ n^{3/2}}\Bigg[\sum_{l_1\neq l_2>k}\re \|x_{i_1l_1}\|_Q\|x_{l_2 i_1}\|_Q^2\left\|\sum_{i_2>k}\be_{l_1}'\re_{k-1}\bd_{ki_1l_1}(z_1)\be_{i_2}\be_{i_2}'\re_{k-1}\bd_{ki_1l_1}(z_2)\be_{i_1}\right\|\\
&\quad\quad\times\Big(\left\|\sum_{i_2>k}\be_{l_2}'\re_{k-1}\bd_{ki_1l_2}(\bar z_1)\be_{i_1}\be_{l_1}'\bd_{ki_1l_2}^{i_1l_1}(\bar z_1)\be_{i_2}\be_{i_2}'\re_{k-1}\bd_{ki_1l_2}(\bar z_2)\be_{i_1}\right\|\\
&\quad\quad+\left\|\sum_{i_2>k}\be_{l_2}'\re_{k-1}\bd_{ki_1l_2}(\bar z_1)\be_{l_1}\be_{i_1}'\bd_{ki_1l_2}^{i_1l_1}(\bar z_1)\be_{i_2}\be_{i_2}'\re_{k-1}\bd_{ki_1l_2}(\bar z_2)\be_{i_1}\right\|\Big)\Bigg]\\
\le&\frac C{ n}+\frac C{ n^{3/2}}\(\sum_{l_1>k}\re \|x_{i_1l_1}\|_Q^2\)^{1/2}\(\sum_{l_2>k}\re\|x_{l_2 i_1}\|_Q^4\)^{1/2}\le Cn^{-1/2}
\end{align*}
where $\bw_{ki_1l_2}^{i_1l_1}=\bw_{ki_1l_2}-\frac1{\sqrt n}\delta_{i_1l_1}\(x_{i_1l_1}\be_{i_1}\be_{l_1}'+x_{l_1i_1}\be_{l_1}\be_{i_1}'\)$ and $\bd_{ki_1l_2}^{i_1l_1}(z)=\(\bw_{ki_1l_2}^{i_1l_1}-z\bi_{2n-2}\)^{-1}$.

Hence, we find that
\begin{align*}
&\sum_{i_2>k}\rtr\(\be_{i_1}'\ba_{k1}(z_1)\be_{i_2}\be_{i_2}'\re_{k-1}\bd_k(z_2)\be_{i_1}\)\\
=&-\frac{m(z_2)}{ n}\sum_{l ,i_2>k}\delta_{i_1l}\rtr\(\be_l'\re_{k-1}\bd_{ki_1l}(z_1)\be_{i_2}\be_{i_2}'\re_{k-1}\bd_{ki_1l}(z_2)
\be_l\)+o_{L_1}(1)\\
=&-\frac{m(z_2)}{ n}\sum_{l ,i_2>k}\delta_{i_1l}\rtr\(\be_l'\re_{k-1}\bd_{k}(z_1)\be_{i_2}\be_{i_2}'\re_{k-1}\bd_{k}(z_2)
\be_l\)+o_{L_1}(1).
\end{align*}
Therefore, we obtain
\begin{align*}
z_1S_k
=&-2{m(z_2)}\(1-\frac kn\)-\frac{m(z_2)}{n}\(1-\frac kn\)\sum_{i_2 , l>k}\rtr\(\be_l'\re_{k-1}\bd_{k}(z_1)\be_{i_2}\be_{i_2}'\re_{k-1}\bd_{k}(z_2)
\be_l\)\\
&-\frac{m(z_1)}{n}\sum_{i_1,i_2>k}\rtr\(\be_{i_1}'\re_{k-1}\bd_{k}(z_1)\be_{i_2}\be_{i_2}'\re_{k-1}\bd_k(z_2)\be_{i_1}\)+o_{L_1}(1)
\end{align*}
which implies that
\begin{align*}
{z_1} S_k
=&-2{m(z_2)}\(1-\frac kn\)-{m(z_2)}\(1-\frac kn\)S_k-{m(z_1)}S_k+o_{L_1}(1).
\end{align*}
Recalling that $z+m(z)=-1/m(z)$, it follows that
\begin{align*}
S_k
=&\frac{2\(1-k/n\)m(z_1)m(z_2)}{1-\(1-k/n\){m(z_1)m(z_2)}}+o_{L_1}(1).
\end{align*}
This yields
\begin{align*}
\lim_{n\to\infty}\frac1n\sum_{k=1}^n S_k=&\lim_{n\to\infty}\frac1n\sum_{k=1}^n \frac{2\(1-k/n\)m(z_1)m(z_2)}{1-\(1-k/n\){m(z_1)m(z_2)}}\\
=&\int_0^1\frac{2tm(z_1)m(z_2)}{1-t{m(z_1)m(z_2)}}dt\\
=&-2\(1-\frac1{m(z_1)m(z_2)}\log\(1-{m(z_1)m(z_2)}\)\).
\end{align*}
Finally, $\sum_{k=1}^n\re_k\left[\re_{k-1}\phi_k(z_1)\re_{k-1}\phi_k(z_2)\right]$ converges in probability to
\begin{align*}
m(z_1)m(z_2)\(\sigma^2-1/2\)+\frac{2M-{3}}{4}m^2(z_1)m^2(z_2)+\frac1{2}\log\(1-{m(z_1)m(z_2)}\).
\end{align*}
Denoting $\Gamma_n(z_j,z_l)\xrightarrow{i.p.}\Gamma(z_j,z_l)$, we conclude that
\begin{align*}
\Gamma(z_1,z_2)=m'(z_1)m'(z_2)\(\sigma^2-1/2+\({2M-{3}}\)m(z_1)m(z_2)+\frac1{2\(1-{m(z_1)m(z_2)}\)^2}\).
\end{align*}

\subsubsection{The Proof of $(\ref{al1})$ for $j=l,r,0$}\label{jl0}
For any $z\in\mathcal{C}_0$, we get
\begin{align*}
\lj m_n(z)I\(B_n^c\)\rj\le2/(a-2)\quad{\rm and}\quad\lj m(z)\rj\le1/(a-2).
\end{align*}
Thus, it follows that
\begin{align*}
&\lim_{v_0\downarrow0}\limsup_{n\to\infty}\re\lj\int_{\mathcal{C}_0}M_n(z)I\(B_n^c\)dz\rj\\
\le&\lim_{v_0\downarrow0}\limsup_{n\to\infty}\int_{\mathcal{C}_0}\re\lj M_n(z)I\(B_n^c\)\rj dz\\
\le&\limsup_{n\to\infty}Cn/(a-2)\varepsilon_nn^{-1}=\limsup_{n\to\infty}C\varepsilon_n/(a-2)=0.
\end{align*}

Note that
\begin{align*}
&\lim_{v_0\downarrow0}\limsup_{n\to\infty}\re\lj\int_{\mathcal{C}_j}M_n(z)I\(B_n^c\)dz\rj^2\\
\le&\lim_{v_0\downarrow0}\limsup_{n\to\infty}\int_{\mathcal{C}_j}\re\lj M_n(z)-\re M_n(z)\rj^2 dz+\lim_{v_0\downarrow0}\limsup_{n\to\infty}\int_{\mathcal{C}_j}\lj \re M_n(z)-\re M(z)\rj^2 dz\\
&+\lim_{v_0\downarrow0}\limsup_{n\to\infty}\int_{\mathcal{C}_j}\lj \re M(z)\rj^2 dz\\
\triangleq&\mathcal{J}_1+\mathcal{J}_2+\mathcal{J}_3.
\end{align*}
From the fact $\limsup_{z\in\mathcal{C}_n}\lj \re M_n(z)-\re M(z)\rj\to 0$, we get
\begin{align*}
\mathcal{J}_2\to0.
\end{align*}
Considering $\re M(z)$ is continuous, it follows that
\begin{align*}
\mathcal{J}_3\le C\lim_{v_0\downarrow0}\int_{\mathcal{C}_j}dz\to0.
\end{align*}
Recalling that
\begin{align*}
M_n(z)-\re M_n(z)=&-\frac1{2n}\sum_{k=1}^n\re_{k-1}b_k(z)\(\rtr\bigg(\bq_k^*\bd_k^2(z)\bq_k\bigg)-\rtr\bd_k^2(z)\)\\
&-\frac18\sum_{k=1}^n\(\re_{k-1}-\re_k\)\(a_{k1}+a_{k2}+a_{k3}\),
\end{align*}
one has by (\ref{eq1}) and Lemma \ref{cll4}
\begin{align*}
&\limsup_{z\in\mathcal{C}_n}\re\lj M_n(z)-\re M_n(z)\rj^2\le\frac C{n^2}\limsup_{z\in\mathcal{C}_n}\sum_{k=1}^n\re\lj b_k(z)\(\rtr\bigg(\bq_k^*\bd_k^2(z)\bq_k\bigg)-\rtr\bd_k^2(z)\)\rj^2\\
&\quad\quad\quad\quad\quad\quad\quad+C\limsup_{z\in\mathcal{C}_n}\sum_{k=1}^n\re\(\lj a_{k1}\rj^2+\lj a_{k2}\rj^2+\lj a_{k3}\rj^2\)\\
&\quad\quad\le\frac C{n^2}\limsup_{z\in\mathcal{C}_n}\sum_{k=1}^n\re\left\|\bd_k(z)\right\|^4+C\limsup_{z\in\mathcal{C}_n}\sum_{k=1}^n\re\(\lj a_{k1}\rj^2+\lj a_{k2}\rj^2+\lj a_{k3}\rj^2\)\\
&\quad\quad\le\frac C{n}+C\limsup_{z\in\mathcal{C}_n}\sum_{k=1}^n\re\(\lj a_{k1}\rj^2+\lj a_{k2}\rj^2+\lj a_{k3}\rj^2\)\le C.
\end{align*}
Therefore, we obtain
\begin{align*}
&\lim_{v_0\downarrow0}\limsup_{n\to\infty}\re\lj\int_{\mathcal{C}_j}M_n(z)I\(B_n^c\)dz\rj^2
\le C\lim_{v_0\downarrow0}\int_{\mathcal{C}_j} dz+o(1)\to 0.
\end{align*}

\subsubsection{Tightness of the Process $M_n(z)-\re M_n(z)$}

We proceed to prove tightness of the sequence of random functions $M_n(z)-\re M_n(z)$. Using Theorem 12.3 of Billingsley \cite{billingsley2009convergence}, it suffices to show that for
$z_1,z_2\in\mathbb{C}_0$
\begin{align}\label{eq4}
\frac{\re\lj \bigg(M_n(z_1)-M_n(z_2)\bigg)-\bigg(\re M_n(z_1)-\re M_n(z_2)\bigg)\rj^2}{\lj z_1-z_2\rj^2}
\end{align}
is finite. Let $b_k(z)=\frac1{z+\frac1{2n}\rtr\bd_k(z)}$, then

 By (\ref{eq2}) and
\begin{align*}
\rtr\zeta_k(z)=2b_k(z)+\frac12b_k(z)g_k(z)\rtr\zeta_k(z),
\end{align*} it yields
\begin{align*}
&\re\lj \bigg(M_n(z_1)-M_n(z_2)\bigg)-\bigg(\re M_n(z_1)-\re M_n(z_2)\bigg)\rj^2=\frac14\sum_{k=1}^n\re\lj\gamma_k(z_1)-\gamma_k(z_2)\rj^2
\end{align*}
and
\begin{align*}
&\gamma_k(z_1)-\gamma_k(z_2)=-\frac12\(\re_{k-1}-\re_k\)\rtr\bigg(\bi_2+\frac1n\bq_k^*\bd_k^2(z_1)\bq_k\bigg)\rtr\zeta_k(z_1)-\gamma_k(z_2)\\
=&-\frac12\(\re_{k-1}-\re_k\)\rtr\bigg(\bi_2+\frac1n\bq_k^*\bd_k^2(z_2)\bq_k\bigg)\rtr\zeta_k(z_1)-\gamma_k(z_2)\\
&-\frac1{2n}\(\re_{k-1}-\re_k\)\left[\rtr\bq_k^*\bigg(\bd_k^2(z_1)-\bd_k^2(z_2)\bigg)\bq_k-\rtr\(\bd_k^2(z_1)-\bd_k^2(z_2)\)\right]\rtr\zeta_k(z_1)\\
&-\frac1{2n}\(\re_{k-1}-\re_k\)\rtr\bigg(\bd_k^2(z_1)-\bd_k^2(z_2)\bigg)\rtr\zeta_k(z_1)\\
=&-4\(\re_{k-1}-\re_k\)\left[g_k(z_1)-g_k(z_2)\right]\rtr\zeta_k(z_1)\rtr\bigg(\bi_2+\frac1n\bq_k^*\bd_k^2(z_2)\bq_k\bigg)\zeta_k(z_2)\\
&+8\left[z_1-z_2\right]\(\re_{k-1}-\re_k\)\rtr\zeta_k(z_1)\rtr\bigg(\bi_2+\frac1n\bq_k^*\bd_k^2(z_2)\bq_k\bigg)\zeta_k(z_2)\\
&+\frac4n\(\re_{k-1}-\re_k\)\rtr\zeta_k(z_1)\rtr\left[\bd_k(z_1)-\bd_k(z_2)\right]\rtr\bigg(\bi_2+\frac1n\bq_k^*\bd_k^2(z_2)\bq_k\bigg)\zeta_k(z_2)\\
&-\frac1{2n}\(\re_{k-1}-\re_k\)\left[\rtr\bq_k^*\bigg(\bd_k^2(z_1)-\bd_k^2(z_2)\bigg)\bq_k-\rtr\(\bd_k^2(z_1)-\bd_k^2(z_2)\)\right]\rtr\zeta_k(z_1)\\
&-\frac1{4n}\(\re_{k-1}-\re_k\)b_k(z_1)g_k(z_1)\rtr\bigg(\bd_k^2(z_1)-\bd_k^2(z_2)\bigg)\rtr\zeta_k(z_1)\\
=&-4\(\re_{k-1}-\re_k\)\left[g_k(z_1)-g_k(z_2)\right]\rtr\zeta_k(z_1)\rtr\bigg(\bi_2+\frac1n\bq_k^*\bd_k^2(z_2)\bq_k\bigg)\zeta_k(z_2)\\
&+8\left[z_1-z_2\right]\(\re_{k-1}-\re_k\)\rtr\zeta_k(z_1)\rtr\bigg(\bi_2+\frac1n\bq_k^*\bd_k^2(z_2)\bq_k\bigg)\zeta_k(z_2)\\
&+\frac8n\(\re_{k-1}-\re_k\)b_k(z_1)\rtr\left[\bd_k(z_1)-\bd_k(z_2)\right]\rtr\bigg(\bi_2+\frac1n\bq_k^*\bd_k^2(z_2)\bq_k\bigg)\zeta_k(z_2)\\
&+\frac2n\(\re_{k-1}-\re_k\)b_k(z_1)g_k(z_1)\rtr\zeta_k(z_1)\rtr\left[\bd_k(z_1)-\bd_k(z_2)\right]\rtr\bigg(\bi_2+\frac1n\bq_k^*\bd_k^2(z_2)\bq_k\bigg)\zeta_k(z_2)\\
&-\frac1{2n}\(\re_{k-1}-\re_k\)\left[\rtr\bq_k^*\bigg(\bd_k^2(z_1)-\bd_k^2(z_2)\bigg)\bq_k-\rtr\(\bd_k^2(z_1)-\bd_k^2(z_2)\)\right]\rtr\zeta_k(z_1)\\
&-\frac1{4n}\(\re_{k-1}-\re_k\)b_k(z_1)g_k(z_1)\rtr\bigg(\bd_k^2(z_1)-\bd_k^2(z_2)\bigg)\rtr\zeta_k(z_1)\\
=&-4\(\re_{k-1}-\re_k\)\left[g_k(z_1)-g_k(z_2)\right]\rtr\zeta_k(z_1)\rtr\bigg(\bi_2+\frac1n\bq_k^*\bd_k^2(z_2)\bq_k\bigg)\zeta_k(z_2)\\
&+8\left[z_1-z_2\right]\(\re_{k-1}-\re_k\)\rtr\zeta_k(z_1)\rtr\bigg(\bi_2+\frac1n\bq_k^*\bd_k^2(z_2)\bq_k\bigg)\zeta_k(z_2)\\
&+\frac8{n^2}\(\re_{k-1}-\re_k\)b_k(z_1)b_k(z_2)\rtr\left[\bd_k(z_1)-\bd_k(z_2)\right]\bigg(\rtr\bq_k^*\bd_k^2(z_2)\bq_k-\rtr\bd_k^2(z_2)\bigg)\\
&+\frac2n\(\re_{k-1}-\re_k\)b_k(z_1)b_k(z_2)g_k(z_2)\rtr\left[\bd_k(z_1)-\bd_k(z_2)\right]\rtr\bigg(\bi_2+\frac1n\bq_k^*\bd_k^2(z_2)\bq_k\bigg)\rtr\zeta_k(z_2)\\
&+\frac2n\(\re_{k-1}-\re_k\)b_k(z_1)g_k(z_1)\rtr\zeta_k(z_1)\rtr\left[\bd_k(z_1)-\bd_k(z_2)\right]\rtr\bigg(\bi_2+\frac1n\bq_k^*\bd_k^2(z_2)\bq_k\bigg)\zeta_k(z_2)\\
&-\frac1{2n}\(\re_{k-1}-\re_k\)\left[\rtr\bq_k^*\bigg(\bd_k^2(z_1)-\bd_k^2(z_2)\bigg)\bq_k-\rtr\(\bd_k^2(z_1)-\bd_k^2(z_2)\)\right]\rtr\zeta_k(z_1)\\
&-\frac1{4n}\(\re_{k-1}-\re_k\)b_k(z_1)g_k(z_1)\rtr\bigg(\bd_k^2(z_1)-\bd_k^2(z_2)\bigg)\rtr\zeta_k(z_1)\\
\triangleq&d_{k1}+d_{k2}+d_{k3}+d_{k4}+d_{k5}+d_{k6}+d_{k7}.
\end{align*}
It is obvious by Lemma \ref{cll3} that
\begin{align}\label{eq3}
\lj\rtr\bigg(\bi_2+\frac1n\bq_k^*\bd_k^2(z_2)\bq_k\bigg)\zeta_k(z_2)\rj=2\lj\frac{\rtr\bigg(\bi_2+\frac1n\bq_k^*\bd_k^2(z_2)\bq_k\bigg)}
{\Im\(2z+n^{-1}\rtr\(\bq_k^*\bd_k(z)\bq_k\)\)}\rj\le\frac2{v_0}.
\end{align}
Using (\ref{al11}), (\ref{eq1}), and (\ref{eq3}), we have
\begin{align*}
\sum_{k=1}^n\re\lj d_{k1}\rj^2\le& Cv_0^{-4}\sum_{k=1}^n\re\lj g_k(z_1)-g_k(z_2)\rj^2\\
=& \frac C{v_0^{4}n^2}\sum_{k=1}^n\re\lj\rtr\bq_k^*\bigg(\bd_k(z_1)-\bd_k(z_2)\bigg)\bq_k-\rtr\(\bd_k(z_1)-\bd_k(z_2)\)\rj^2\\
=& \frac{ C|z_1-z_2|^2}{v_0^{4}n^2}\sum_{k=1}^n\re\lj\rtr\bigg(\bq_k^*\bd_k(z_1)\bd_k(z_2)\bq_k\bigg)-\rtr\(\bd_k(z_1)\bd_k(z_2)\)\rj^2\\
\le& \frac{ C|z_1-z_2|^2}{v_0^{8}}
\end{align*}
and
\begin{align*}
\sum_{k=1}^n\re\lj d_{k2}\rj^2\le&256|z_1-z_2|^2\re\lj\rtr\zeta_k(z_1)\rtr\bigg(\bi_2+\frac1n\bq_k^*\bd_k^2(z_2)\bq_k\bigg)\zeta_k(z_2)\rj^2\\
\le&\frac{ C|z_1-z_2|^2}{v_0^{4}}.
\end{align*}
Applying
\begin{align*}
\re\lj\rtr\bq_k^*\bd_k^2(z)\bq_k-\rtr\bd_k^2(z)\rj^2\le Cnv_0^{-4}
\end{align*}
and
\begin{align*}
\re\lj g_k(z)\rj^2\le\frac Cn\re\|x_{kk}\|_Q^2+\frac C{n^2}\re\lj\rtr\bq_k^*\bd_k(z)\bq_k-\rtr\bd_k(z)\rj^2 \le\frac Cn,
\end{align*}
we have
\begin{align*}
\sum_{k=1}^n\re\lj d_{k3}+d_{k4}+d_{k5}\rj^2\le&\frac C{n^3}\sum_{k=1}^n\re\lj\rtr\left[\bd_k(z_1)-\bd_k(z_2)\right]\rj^2\\
\le&\frac{ C|z_1-z_2|^2}{n^3}\sum_{k=1}^n\re\left\|\bd_k(z_1)\bd_k(z_2)\right\|^2\le \frac{ C|z_1-z_2|^2}{n^2}.
\end{align*}
and
\begin{align*}
\sum_{k=1}^n\re\lj d_{k6}+d_{k7}\rj^2\le&\frac C{n^2}\sum_{k=1}^n\re\rtr\bigg(\bd_k^2(z_1)-\bd_k^2(z_2)\bigg)\bigg(\bd_k^2(\bar z_1)-\bd_k^2(\bar z_2)\bigg)\\
\le&\frac C{nv_0^2}\sum_{k=1}^n\re\left\|\bd_k(z_1)-\bd_k(z_2)\right\|^2\le \frac{ C|z_1-z_2|^2}{v_0^6}.
\end{align*}
Hence, (\ref{eq4}) is proved and the tightness of the process $M_n(z)-\re M_n(z)$ holds.

\section{appendix}

\begin{lemma}[Theorem A.37 in \cite{bai2010spectral}]\label{cll1}
Suppose $\ba$ and $\mb B$ are two $n\times n$ matrices and $\lambda_k$ and $\delta_k,k=1,cdots,n$, denote their singular values. If the singular values are arranged in descending order, then we have
\begin{align*}
\sum_{k=1}^n|\lambda_k-\delta_k|^2\le\rtr\left[\(\ba-\mb B\)\(\ba-\mb B\)^*\right].
\end{align*}
\end{lemma}

\begin{lemma}[Corollary 2.5 in \cite{yin2014rate}]\label{cll3}
Under the conditions of Theorem \ref{th1}, we have
\begin{itemize}
\item[(1)] $\bd(z)$ and $\bd_k(z)$ are all Type-$I$ matrices,
\item[(2)] $\vk$ and $\zk$ are all scalar matrices.
\end{itemize}
\end{lemma}

\begin{lemma}[See appendix A.1.4 in \cite{bai2010spectral}]\label{cll2}
Suppose that the matrix $\boldsymbol\Sigma$ has the partition as given by $\begin{pmatrix}\boldsymbol\Sigma_{11}&\boldsymbol\Sigma_{12}\\ \boldsymbol\Sigma_{21}&\boldsymbol\Sigma_{22}\end{pmatrix}$. If $\boldsymbol\Sigma$ and $\boldsymbol\Sigma_{11}$ are  nonsingular, then the inverse of $\boldsymbol\Sigma$ has the form
\begin{align*}
\boldsymbol\Sigma^{-1}=\begin{pmatrix}
\boldsymbol\Sigma_{11}^{-1}+\boldsymbol\Sigma_{11}^{-1}\boldsymbol\Sigma_{12}\boldsymbol\Sigma_{22.1}^{-1}\boldsymbol\Sigma_{21}\boldsymbol\Sigma_{11}^{-1}&
-\boldsymbol\Sigma_{11}^{-1}\boldsymbol\Sigma_{12}\boldsymbol\Sigma_{22.1}^{-1}\\
-\boldsymbol\Sigma_{22.1}^{-1}\boldsymbol\Sigma_{21}\boldsymbol\Sigma_{11}^{-1}&\boldsymbol\Sigma_{22.1}^{-1}
\end{pmatrix}
\end{align*}
where $\boldsymbol\Sigma_{22.1}=\boldsymbol\Sigma_{22}-\boldsymbol\Sigma_{21}\boldsymbol\Sigma_{11}^{-1}\boldsymbol\Sigma_{12}$.
\end{lemma}

\begin{lemma}[Lemma 2.18 in \cite{yin2014rate}]\label{cll4}
Let ${\mathbf A}$ be a $2n\times2n$ non-random matrix and ${\mathbf X}=(x_1',\cdots,x_n')'$ be a random quaternion vector of independent entries, where for $1\leq j\leq n$,
 $$x_j=\left(
        \begin{array}{cc}
          e_j+f_j\cdot i & c_j+d_j\cdot i\\
          -c_j+d_j\cdot i & e_j-f_j\cdot i \\
        \end{array}
      \right)
 =\left(
        \begin{array}{cc}
          \alpha_j & \beta_j \\
          -\bar \beta_j & \bar \alpha_j \\
        \end{array}
      \right)
.$$
Assume that $\re x_j=\mb 0
$, $\re\left\|x_j\right\|_{Q}^2=1$, and $\re\left\|x_j\right\|_{Q}^l\leq \phi_l$. Then, for any $q\geq 1$, we have
$$\re\left|\rtr{\mathbf X}^*{\mathbf A}{\mathbf X}-\rtr {\mathbf A} \right|^q \leq C_q \(\(\phi_4\rtr\({\mathbf A}{\mathbf A^*}\)\)^{q/2}+\phi_{2q}\rtr\({\mathbf A}{\mathbf A^*}\)^{q/2}\),$$
where $C_p$ is a constant depending on $p$ only.
\end{lemma}

\begin{lemma}[Burkholder's inequality ]\label{cll5}
Let $\{ {{\mathbf X}_k}\} $ be a complex martingale difference sequence with respect to the increasing $\sigma$-field. Then, for $p > 1,$ $${\rm E}{\left| {\sum_k {{{\mathbf  X}_k}} } \right|^p} \le {K_p}{\rm E}{\left({\sum_k{\left| {{{\mathbf  X}_k}} \right|} ^2}\right)^{p/2}}.$$
\end{lemma}

\begin{lemma}\label{cll6}
Let ${\mb A}$ and ${\mb B}$ be two $2n\times2n$ non-random Type-\uppercase\expandafter{\romannumeral1} matrices while ${\mathbf X}=(x_1',\cdots,x_n')'$ be a random quaternion vector of independent entries, where for $1\leq j\leq n$,
 $$x_j=\left(
        \begin{array}{cc}
          e_j+f_j\cdot i & g_j+h_j\cdot i\\
          -g_j+h_j\cdot i & e_j-f_j\cdot i \\
        \end{array}
      \right)
 =\left(
        \begin{array}{cc}
          \alpha_j & \beta_j \\
          -\bar \beta_j & \bar \alpha_j \\
        \end{array}
      \right).
$$
Assume that $\re x_j=\mb 0$, $\re\|x_j\|_Q^4=M_4$ and $\re\left\|x_j\right\|_{Q}^2=1$ such that
$$\re|\alpha_j|^2=1/2,\quad \re|\beta_j|^2=1/2,\quad\re\alpha_j^2=0,\quad \re\beta_j^2=0,\quad\re(\alpha_j\beta_j)=\re(\alpha_j\bar\beta_j)=0.$$ Then, splitting $\mb A$ and $\mb B$ into $2\times 2$ blocks and denoting $\mb A=\(\mathbbm{a}_{j,k}\)$ and $\mb B=\(\mathbbm{b}_{j,k}\)$, we have
\begin{align*}
  &\re\(\rtr{\mathbf X}^*{\mathbf A}{\mathbf X}-\rtr {\mathbf A}\)\(\rtr{\mathbf X}^*{\mathbf B}{\mathbf X}-\rtr {\mathbf B}\)\\\notag
  =&\re \(\|x_1\|_Q^4-\frac{3}{2}\)\sum_{j}\rtr \mathbbm{a}_{j,j}\rtr \mathbbm{b}_{j,j}+\rtr \mb A\mb B.
\end{align*}
\end{lemma}
\begin{proof}
  At first, obviously
  \begin{align*}
  &\re\(\rtr{\mathbf X}^*{\mathbf A}{\mathbf X}-\rtr {\mathbf A}\)\(\rtr{\mathbf X}^*{\mathbf B}{\mathbf X}-\rtr {\mathbf B}\)\\\notag
  =&\re\(\sum_{j,k}\rtr \mathbbm{a}_{j,k} x_k x_j^*\)\(\sum_{j,k}\rtr \mathbbm{b}_{j,k} x_k x_j^*\)-\rtr {\mathbf A}\rtr {\mathbf B}\\\notag
  =&\re\Bigg(\sum_{j}\rtr \mathbbm{a}_{j,j} x_j x_j^* \rtr \mathbbm{b}_{j,j} x_j x_j^*+\sum_{j\neq k}\rtr \mathbbm{a}_{j,k}x_k x_j^* \rtr \mathbbm{b}_{j,k} x_k x_j^*\\\notag
  &+\sum_{j\neq k}\rtr \mathbbm{a}_{j,k} x_k x_j^*\rtr \mathbbm{b}_{k,j}x_j x_k^* +\sum_{j\neq k}\rtr \mathbbm{a}_{j,j} x_j x_j^*\rtr \mathbbm{b}_{k,k} x_k x_k^*\Bigg)-\rtr {\mathbf A}\rtr {\mathbf B}.\notag
\end{align*}

Next, we will compute the above expression term by term. Notice that both ${\mb A}$ and ${\mb B}$ are Type-\uppercase\expandafter{\romannumeral1} matrices, we know that for all $1\leq j\leq n$, $\mathbbm{a}_{j,j}$ and $\mathbbm{b}_{j,j}$ are all Type-T matrices. Thus we obtain that
\begin{align*}
  \re\sum_{j}\rtr \mathbbm{a}_{j,j} x_j x_j^* \rtr \mathbbm{b}_{j,j} x_j x_j^*=\re \|x_1\|_Q^4\sum_{j}\rtr \mathbbm{a}_{j,j}\rtr \mathbbm{b}_{j,j},
\end{align*}
and
\begin{align*}
  &\re\sum_{j\neq k}\rtr \mathbbm{a}_{j,j} x_j x_j^*\rtr \mathbbm{b}_{k,k}x_k x_k^* =\(\re\|x_1\|_Q^2\)^2\sum_{j\neq k}\rtr \mathbbm{a}_{j,j}\rtr \mathbbm{b}_{k,k}\\
  =&\(\re\|x_1\|_Q^2\)^2\sum_{j,k}\rtr \mathbbm{a}_{j,j}\rtr \mathbbm{b}_{k,k}-\(\re\|x_1\|_Q^2\)^2\sum_{j}\rtr \mathbbm{a}_{j,j}\rtr \mathbbm{b}_{j,j}\\
    =&\rtr \mb A\rtr \mb B-\sum_{j}\rtr \mathbbm{a}_{j,j}\rtr \mathbbm{b}_{j,j}.
\end{align*}
Moreover, denote by $y_{j,k}=\left(
                               \begin{array}{cc}
                                 \varpi_{j,k} & \omega_{j,k} \\
                                 -\bar \omega_{j,k} & \bar\varpi_{j,k} \\
                               \end{array}
                             \right)$
the $x_k x_j^*$
where
$$\varpi_{j,k}=
     \alpha_k\bar\alpha_j+\beta_k\bar \beta_j \qquad {\rm and} \qquad \omega_{j,k}=-\alpha_k\beta_j+\beta_k\alpha_j.$$
From the above expressions and conditions, we have for $j\neq k$,
$$\re\varpi_{j,k}^2=\re\omega_{j,k}^2=\re\varpi_{j,k}\omega_{j,k}=\re\bar\varpi_{j,k}\omega_{j,k}=0,\quad \re|\varpi_{j,k}|^2=\re|\omega_{j,k}|^2=\frac{1}{2}.$$
By calculating, one finds
\begin{align*}
  &\re\sum_{j\neq k}\rtr \mathbbm{a}_{j,k} x_k x_j^*\rtr \mathbbm{b}_{j,k}x_k x_j^* \\
  =& \re\sum_{j\neq k}\(a_{j,k}^{lu}\varpi_{j,k}-a_{j,k}^{ru}\bar\omega_{j,k}+a_{j,k}^{ld}\omega_{j,k}+a_{j,k}^{rd}\bar\varpi_{j,k}\)
  \(b_{j,k}^{lu}\varpi_{j,k}-b_{j,k}^{ru}\bar\omega_{j,k}+b_{j,k}^{ld}\omega_{j,k}+b_{j,k}^{rd}\bar\varpi_{j,k}\)\\
  =&\frac{1}{2}\sum_{j\neq k}\(a_{j,k}^{lu}b_{j,k}^{rd}-a_{j,k}^{ru}b_{j,k}^{ld}-a_{j,k}^{ld}b_{j,k}^{ru}+a_{j,k}^{rd}b_{j,k}^{lu}\)\\
  =&\frac{1}{2}\sum_{j\neq k}\rtr\(\left(
                           \begin{array}{cc}
                             a_{j,k}^{lu} & a_{j,k}^{ru}\\
                             a_{j,k}^{ld} & a_{j,k}^{rd} \\
                           \end{array}
                         \right)\left(
                                  \begin{array}{cc}
                                    b_{j,k}^{rd} & -b_{j,k}^{ru} \\
                                    -b_{j,k}^{ld} & b_{j,k}^{lu} \\
                                  \end{array}
                                \right)
                         \)
  =\frac{1}{2}\sum_{j\neq k}\rtr\mathbbm{a}_{j,k}\mathbbm{b}_{k,j}=\frac{1}{2}\rtr\mb A\mb B-\frac{1}{2}\sum_{j}\rtr\mathbbm{a}_{j,j}\mathbbm{b}_{j,j}\\
  =&\frac{1}{2}\rtr\mb A\mb B-\frac{1}{4}\sum_{j}\rtr\mathbbm{a}_{j,j}\rtr\mathbbm{b}_{j,j}.
\end{align*}
Here the last third equality used the property of Type-\uppercase\expandafter{\romannumeral1} matrices.

By the same argument, we have
\begin{align*}
  &\re\sum_{j\neq k}\rtr \mathbbm{a}_{j,k} x_k x_j^*\rtr \mathbbm{b}_{k,j}x_j x_k^* \\
  =& \re\sum_{j\neq k}\(a_{j,k}^{lu}\varpi_{j,k}-a_{j,k}^{ru}\bar\omega_{j,k}+a_{j,k}^{ld}\omega_{j,k}+a_{j,k}^{rd}\bar\varpi_{j,k}\)
  \(b_{k,j}^{lu}\bar\varpi_{j,k}+b_{k,j}^{ru}\bar\omega_{j,k}-b_{k,j}^{ld}\omega_{j,k}+b_{k,j}^{rd}\varpi_{j,k}\)\\
  =&\frac{1}{2}\sum_{j\neq k}\(a_{j,k}^{lu}b_{k,j}^{lu}+a_{j,k}^{ru}b_{k,j}^{ld}+a_{j,k}^{ld}b_{k,j}^{ru}+a_{j,k}^{rd}b_{k,j}^{rd}\)\\
  =&\frac{1}{2}\sum_{j\neq k}\rtr\(\left(
                           \begin{array}{cc}
                             a_{j,k}^{lu} & a_{j,k}^{ru}\\
                             a_{j,k}^{ld} & a_{j,k}^{rd} \\
                           \end{array}
                         \right)\left(
                                  \begin{array}{cc}
                                    b_{k,j}^{lu} & b_{k,j}^{ru} \\
                                    b_{k,j}^{ld} & b_{k,j}^{rd} \\
                                  \end{array}
                                \right)
                         \)
  =\frac{1}{2}\sum_{j\neq k}\rtr\mathbbm{a}_{j,k}\mathbbm{b}_{k,j}=\frac{1}{2}\rtr\mb A\mb B-\frac{1}{2}\sum_{j}\rtr\mathbbm{a}_{j,j}\mathbbm{b}_{j,j}\\
  =&\frac{1}{2}\rtr\mb A\mb B-\frac{1}{4}\sum_{j}\rtr\mathbbm{a}_{j,j}\rtr\mathbbm{b}_{j,j}.
\end{align*}
    Combining the argument above, we finally get that
    \begin{align*}
  &\re\(\rtr{\mathbf X}^*{\mathbf A}{\mathbf X}-\rtr {\mathbf A}\)\(\rtr{\mathbf X}^*{\mathbf B}{\mathbf X}-\rtr {\mathbf B}\)=\re \(\|x_1\|_Q^4-\frac{3}{2}\)\sum_{j}\rtr \mathbbm{a}_{j,j}\rtr \mathbbm{b}_{j,j}+\rtr \mb A\mb B.
\end{align*}
The proof of this lemma is complete.
\end{proof}

\begin{lemma}[Theorem 35.12 of Billingsley(1995)]\label{cll7}
Suppose for each $n$ $Y_{n1},Y_{n2},\cdots,Y_{nr_n}$ is a real martingale difference sequence with respect to the increasing $\sigma$-field $\{\mathcal{F}_{nj}\}$ having second moments. If as $n\to\infty$,
\begin{align*}
&(i)\quad\quad\quad\quad\qquad\qquad\sum_{j=1}^{r_n}\re(Y_{nj}^2|\mathcal{F}_{n,j-1})\xrightarrow{i.p.}\sigma^2,\qquad\qquad\qquad\qquad\qquad\qquad\\
&{\rm where}\ \sigma^2\ {\rm is\ a\ positive\ constant}, \ {\rm and\ for\ each}\ \varepsilon\ge0,\\
&(ii)\quad\quad\quad\quad\qquad\qquad\sum_{j=1}^{r_n}\re(Y_{nj}^2I(|Y_{nj}\ge\varepsilon|))\to0,\qquad\qquad\qquad\qquad\qquad\qquad\\
&{\rm then}\\
&\quad\quad\quad\quad\qquad\qquad\qquad\sum_{j=1}^{r_n}Y_{nj}\xrightarrow{D}N(0,\sigma^2).\qquad\qquad\qquad\qquad
\end{align*}
\end{lemma}

\end{document}